\documentclass[12pt]{amsart}

\usepackage{amssymb,amsthm,amsfonts,latexsym}
\usepackage{amsmath}
\usepackage{color}
\usepackage{enumitem}
\usepackage{pxfonts} 
\usepackage{graphicx}

\usepackage{tikz}
\usepackage{tikz-cd}
\usepackage[backref]{hyperref}
\usepackage{doi}
\usetikzlibrary{positioning}
\input{xy}
\xyoption{all}
\xyoption{poly}
\usepackage[all]{xy}
\allowdisplaybreaks

\theoremstyle{plain}
\newtheorem{theorem}{Theorem}[section]
\newtheorem*{theorem*}{Theorem}
\newtheorem{lemma}[theorem]{Lemma}
\newtheorem{corollary}[theorem]{Corollary}
\newtheorem*{corollary*}{Corollary}
\newtheorem{proposition}[theorem]{Proposition}
\newtheorem{conjecture}[theorem]{Conjecture}
\newtheorem{question}[theorem]{Question}

\theoremstyle{definition}
\newtheorem{example}[theorem]{Example}
\newtheorem{definition}[theorem]{Definition}

\theoremstyle{remark}
\newtheorem{remark}[theorem]{Remark}

\numberwithin{equation}{section}

 \renewcommand\SS{{\mathbb{S}}}
 \newcommand\ZZ{{\mathbb{Z}}}
 
 \def\A{\mathcal{A}}
 
 \def\C{{\mathcal C}}
 \def\D{{\mathcal D}}

 \def\H{{\mathcal H}}

 \def\K{{\mathcal K}}
 \def\L{{\mathcal L}}

 \def\P{{\mathcal P}}
 
 \def\R{{\mathcal R}}
 \def\S{{\mathcal S}}

 \def\X{{\mathcal X}}

 \def\Lk{{\text{Lk}}}
 
 \def\St{\text{St}}  
 
 \def\op{\mathrm{op}}

 \newcommand{\GF}[1]{\mathbb{F}_{#1}}

 \DeclareMathOperator\GL{GL}

 \DeclareMathOperator\Aut{Aut}

 \DeclareMathOperator\id{id}
  
 \def\Im{\mathrm{Im}}

 \DeclareMathOperator{\codim}{codim}
 
 \newcommand{\tq}{\mathrel{{\ensuremath{\: : \: }}}}

 \DeclareMathOperator\Stab{Stab}
 
 \DeclareMathOperator\Ind{Ind}

 \def\groupiso{\cong}

 \newcommand\gen[1]{\left\langle#1\right\rangle}

\renewcommand{\gen}[1]{\langle #1 \rangle}

\def\GG{G}

\def\PD{{\mathcal{PD}}}
\def\OPD{{\mathcal{OPD}}}

\def\OD{{\mathcal{OD}}}
\def\CB{{\rm CB}}
\def\Opp{{\rm Opp}}
\newcommand\OKtwo[1]{{\OPD(\Delta(#1))}}
\newcommand\Ktwo[1]{{\PD(\Delta(#1))}}
\def\Ch{{\rm Ch}}

\def\Conv{{\rm Conv}}
\renewcommand{\S}{\mathcal{S}}

\renewcommand{\K}{\mathcal{K}}

\newcommand{\redH}{\widetilde{H}}
\newcommand{\redm}[1]{{#1}^{\circ}}
\newcommand{\red}[1]{{#1}^{*}}
\newcommand{\proj}{\mathrm{proj}}

\newcommand{\cham}{\mathrm{Cham}\,}
\newcommand{\roots}{\mathrm{Root}\,}
\newcommand{\dcat}{{d_{\circ}}}

\newcommand{\sd}{\mathrm{sd}}
\def\SS{{\mathbb{S}}}

\makeatletter
\newcommand{\ojoin}{\mathbin{\mathpalette\make@circled\star}}
\newcommand{\make@circled}[2]{%
\ooalign{$\m@th#1\smallbigcirc{#1}$\cr\hidewidth$\m@th#1#2$\hidewidth\cr}%
}
\newcommand{\smallbigcirc}[1]{%
  \vcenter{\hbox{\scalebox{0.9}{$\m@th#1\bigcirc$}}}%
}
\makeatother

\author{Kevin I. Piterman}
\address{Vrije Universiteit Brussel\\
Department of Mathematics \& Data Science\\
1050 Brussels, Belgium}
\email{kevin.piterman@vub.be}

\author{John Shareshian}
\address{Washington University in St Louis\\
Department of Mathematics\\
St. Louis, MO 63130, USA
}
\email{jshareshian@wustl.edu}

\author{Volkmar Welker}
\address{Philipps-Universit\"at Marburg \\
Fachbereich Mathematik und Informatik \\
35032 Marburg, Germany}
\email{welker@mathematik.uni-marburg.de}

\begin{document}

\title{Posets of decompositions in spherical buildings}

\begin{abstract}
We propose definitions of the common bases complex, the poset of decompositions, and the poset of partial decompositions for arbitrary spherical buildings.
We show that the poset of decompositions is Cohen-Macaulay, and that the poset of partial decompositions is spherical and homotopy equivalent to the common bases complex.
To prove these results, we rely on the concepts of opposition, Levi spheres, and convexity in buildings.

In particular, our results extend the already known constructions for the linear case (vector spaces) to arbitrary buildings.
As a byproduct, we see that the poset of ordered partial decompositions carries the square of the Steinberg representation.
\end{abstract}

\maketitle

\tableofcontents

\section{Introduction}
\label{sec:intro}

The order complex of the lattice of non-trivial subspaces of a finite-dimensional vector space is well known to be the spherical building of type A (see e.g., \cite{Tits1}).
Recently, simplicial complexes defined 
on the set of non-trivial subspaces 
via constraints on common bases (see e.g., \cite{Rognes,MPW}) or the relative position of the subspaces were studied (see e.g., \cite{LR,BPW24}). Such complexes are often order complexes of partially ordered sets.
The goals of this paper are to provide definitions of such
simplicial complexes and posets that are independent of Lie type, and to provide results on the associated homology groups
and homotopy types. In some cases, the complexes we construct in type A are the same as the previously studied ones.  In other cases, our type A complexes are different than those previously studied but have the same equivariant homotopy type. We start by defining the motivating posets and simplicial complexes and reviewing the relevant literature. 

Let $V$ be an $n$-dimensional vector space over a field $k$.
A partial decomposition of $V$ is a set $\{V_1,\ldots,V_r\}$ of non-zero and proper subspaces of $V$ such that $\gen{V_1,\ldots, V_r} \cong V_1\oplus\cdots\oplus V_r$;
that is, the $V_1,\ldots, V_r$ are in internal direct sum.
The set $\PD(V)$ of all non-empty partial decompositions of $V$ is a poset with order given by refinement; i.e., $\{V_1,\ldots, V_r\}\leq \{W_1,\ldots, W_s\}$ if for all $1\leq i \leq r$ there is $1\leq j\leq s$ such that $V_i\leq W_j$.

In \cite{HHS}, P. Hanlon, P. Hersh and J. Shareshian proved that $\PD(V)$ is Cohen-Macaulay of dimension $2n-3$ if $k$ is a finite field.
In particular, $\PD(V)$ is spherical (see Section \ref{sec:preliminaries} for definitions).
More recently, in \cite{BPW24}, it was proved that for any field $k$,\,$\PD(V)$ has the homotopy type of the common bases complex $\CB(V)$ of $V$, as defined by J. Rognes \cite{Rognes}.
Concretely, $\CB(V)$ is the simplicial complex whose simplices are sets $\sigma$ of non-zero proper subspaces of $V$ for which there exists a basis $B$ of $V$ that each subspace in $\sigma$ is generated by some subset of $B$.
In this case, we say that the elements of $\sigma$ have a common basis.
Rognes showed that, although $\CB(V)$ has dimension $2^n-3$, its homology is concentrated in degrees $\leq 2n-3$, and he conjectured that $\CB(V)$ must indeed be $(2n-4)$-connected.
In \cite{MPW}, J. Miller, P. Patzt and J. Wilson established Rognes' conjecture by showing that there is a continuous map from a $(2n-4)$-connected simplicial complex to $\CB(V)$ that induces an isomorphism on homotopy groups up to degree $2n-4$.
In particular, via \cite{BPW24} this implies that $\PD(V)$ is spherical for any field.

Considering the ordered version $\OPD(V)$ of $\PD(V)$ 
appears to be equally natural.
By definition, $\OPD(V)$ is the poset whose elements are tuples 
$(V_1,\ldots,V_r)$, $r\geq 1$, of distinct subspaces of $V$ such that $\{V_1,\ldots,V_r\} \in \PD(V)$.
The ordering in $\OPD(V)$ is given by refinement consistent with the ordering of the tuples (see Definition \ref{def:ODandOPDVectorSpaces}).
We call the elements of $\OPD(V)$ ordered partial decompositions.
It follows from \cite[Theorem 6.8]{PW25} that $\OPD(V)$ is homotopy equivalent to a wedge of spheres of dimension $2n-3$.
If $V$ is a finite vector space then \cite[Theorem 6.10]{PW25} implies 
that $\OPD(V)$ has the homotopy type of the two-fold join $\Delta * \Delta$ of the Tits building $\Delta$ of $V$.
Using GAP we also verified in some small examples that the top homology group of
$\OPD(V)$, as a $\GL(V)$-module, is the tensor square of the Steinberg module.

Later, we found an alternative approach to determining the 
homotopy type of $\OPD(V)$ using closely related constructions, which inspired the work in this paper.
Let $\OD(V)$ denote the subposet of $\OPD(V)$ consisting of full ordered decompositions $(V_1,\ldots, V_r)$; i.e., those decompositions satisfying
$V \cong V_1\oplus \cdots \oplus V_r$.
Let $X = \Delta \cup \OD(V)$ be the poset on the disjoint union of 
$\Delta$ (regarded as a poset without the empty simplex) and $\OD(V)$ with the following
ordering.
We keep the inclusion ordering among elements of $\Delta$, and the refinement ordering in $\OD(V)$.  No element of $\OD(V)$ lies below any element of $\Delta$. If $\sigma\in \Delta$ and $d\in \OD(V)$, we set $\sigma \prec d$ if and only if there is a basis of $V$ containing a basis for $d$ and a basis for every subspace appearing in $\sigma$.
In Theorem \ref{thm:PDandPDBuildingVectorSpace} we show that
$\OPD(V) \simeq X$.
The poset $X$ in turn is homotopy equivalent to the complex $T^{1,1}(V)$ used in \cite{MPW}, which is shown therein to have
the homotopy type of $\Delta * \Delta$.
This again allows us to conclude that $\OPD(V) \simeq \Delta * \Delta$ for any finite-dimensional vector space $V$.
Since all homotopy equivalences can be seen to be $\GL(V)$-equivariant, it now follows that the top homology group of
$\OPD(V)$ carries the tensor square of the Steinberg representation. 

The poset $\OD(V)$ already appears in work {\cite{LR} of  G.I. Lehrer and L.J. Rylands, who observe that 
$\OD(V)$ can be identified with poset of pairs of opposite parabolics of $\GL(V)$.
\footnote{This observation is motivated by previous work by R. Charney \cite{Charney}, and $\OD(V)$ is commonly known as the Charney complex. See also \cite[Proposition 4.15]{PW25}.}
Recall that two parabolic subgroups are called opposite if their intersection is a Levi complement in both of them.
Indeed, two parabolic subgroups are opposite if and only if they are opposite as simplices of the building $\Delta$ of $V$, which is a notion intrinsically defined for every building (see Subsection \ref{sub:opposition}).
Thus, $\OD(V)$ can be described in terms of simplices of $\Delta$, and we define
\[\, \OD(\Delta) = \big\{\,(\sigma,\sigma') \tq \sigma,\sigma'\in \Delta \text{ are (non-empty) opposite simplices}\,\big\},\]
where the ordering of this poset is inclusion-reversing in each coordinate (so a pair of opposite maximal simplices is an ordered frame of $V$, which is a minimal element of the poset).
Using this identification the ordering between crossed terms in $X = \Delta \cup \OD(\Delta)$ becomes $\sigma\prec (\sigma_1,\sigma_2)$ if there is a basis of $V$ spanning the subspaces from $\sigma\cup \sigma_1\cup \sigma_2$.
This condition is equivalent to  saying that $\sigma,\sigma_1,\sigma_2$ lie in a common apartment of $\Delta$. This shows
that
the poset $\Delta\cup \OD(\Delta)$ can be defined in terms of intrinsic combinatorial properties of the building.

These observations suggest that there may be a more general phenomenon in the context of algebraic groups or even spherical buildings.
In this paper, we study the following questions and answer all of them positively.
Let $\Delta$ be a (spherical) building.
\begin{enumerate}
    \item Can we define a poset $\OPD(\Delta)$ that, up to homotopy, coincides with $\OPD(V)$ when $\Delta$ is the building of the vector space $V$?
    \item Can we similarly define $\PD(\Delta)$? Is it highly connected?
    \item Do we have a common basis complex $\CB(\Delta)$ for $\Delta$ that coincides with $\CB(V)$ when $\Delta$ is the building of the vector space $V$?
    \item If questions (2) and (3) have a positive answer, do we have $\CB(\Delta) \simeq \PD(\Delta)$?
    \item Is $\OPD(\Delta)\simeq \Delta * \Delta$?
\end{enumerate}

While the definitions of $\OPD(\Delta)$ and $\CB(\Delta)$ arise naturally, a suitable definition of $\PD(\Delta)$ is less obvious.
Indeed, as in the case of $\OPD$, Theorem \ref{thm:PDandPDBuildingVectorSpace} shows that $\PD(V)\simeq \Delta \cup \D(V)$, where $\D(V)$ is the poset of full decompositions of $V$, and the ordering between crossed terms in $\Delta \cup \D(V)$ is defined as in the ordered case: if $\sigma\in \Delta$ and $d\in \D(V)$, then we set $\sigma\prec d$ if $\sigma\cup d\in \CB(V)$ ($\sigma$ and $d$ have a common basis).
Our next observation is, that $\D(V)$ can be identified with the poset of (split) Levi subgroups of $\GL(V)$.
In the language of buildings, split Levi subgroups correspond to Levi spheres, as introduced by J.P. Serre in \cite{Serre} (see Subsection \ref{sub:convexity} and Section \ref{sec:leviSpheres}).
If we identify a full decomposition $d\in \D(V)$ with a Levi sphere, then the condition $\sigma \prec d$ becomes ``$\sigma$ and the Levi sphere $d$ lie in a common apartment".
Therefore, for a spherical building $\Delta$, we propose:

\begin{itemize}
    \item $\CB(\Delta) = \{\sigma \tq \sigma$ is a subset of the vertex set of some apartment of $\Delta\}$ (Definition \ref{def:commonBasis}) 
    \item $\D(\Delta) = $ poset of non-empty Levi spheres ordered by reverse inclusion (Definition \ref{def:decompBuilding}).
    \item $\PD(\Delta) = \Delta\cup \D(\Delta)$, with crossed-term ordering $\sigma\prec S$ if there is an apartment containing the Levi sphere $S$ and the simplex $\sigma$ (Definition \ref{def:PDbuildings}).
    \item $\OD(\Delta) = \{(\sigma_1,\sigma_2)\tq \sigma_1,\sigma_2$ are non-empty opposite simplices of $\Delta\}$, with coordinate-wise reverse-inclusion ordering (Definition \ref{def:ODandOPDbuildings}).
    \item $\OPD(\Delta) = \Delta \cup \OD(\Delta)$, with crossed-term ordering $\sigma\prec (\sigma_1,\sigma_2)$ if there is an apartment containing $\sigma,\sigma_1,\sigma_2$ (Definition \ref{def:ODandOPDbuildings}).
\end{itemize}

The main theorems of the paper are the following:

\begin{theorem}
Let $\Delta$ be a spherical building.
There are order-preserving maps
\[ \Gamma: \sd \PD(\Delta) \to \CB(\Delta),\]
\[ \phi :\OPD(\Delta) \to \Delta * \Delta,\]
such that for every group $H$ acting on $\Delta$ by simplicial automorphisms, $\Gamma$ and $\phi$ are $H$-equivariant and induce homotopy equivalences between the fixed point subposets:
\[ \Gamma_H : \PD(\Delta)^H \to \CB(\Delta)^H,\]
\[ \phi_H :\OPD(\Delta)^H \to \Delta^H * \Delta^H.\]
In particular, $\Gamma$ and $\phi$ induce homotopy equivalences and equivariant isomorphisms in (co)homology.
\end{theorem}

Here, $\CB(\Delta)$ and $\Delta$ are regarded as posets via their face posets, and the superscript $H$ means that we are taking the $H$-fixed point subposet.

There is a canonical map $F:\OD(\Delta)\to \D(\Delta)$ that maps a pair of opposite simplices $(\sigma_1,\sigma_2)\in\OD(\Delta)$ to the Levi sphere spanned by $\sigma_1,\sigma_2$ (which is actually their convex hull, see Subsection \ref{sub:convexity}).

In \cite{vH}, A. von Heydebreck showed that $\OD(\Delta)$ is Cohen-Macaulay of dimension $\dim\Delta$.
Using the map $F$, we conclude:

\begin{corollary}
\label{coro:DDeltaCM}
Let $\Delta$ be a spherical building.
Then the poset of Levi spheres $\D(\Delta)$ is Cohen-Macaulay of dimension $\dim \Delta$.
\end{corollary}

We mention that Corollary \ref{coro:DDeltaCM} can potentially be applied to compute the homotopy type of Quillen's $p$-subgroup posets \cite{Qui78} of connected reductive algebraic groups $G$ over finite fields $\GF{q}$, where $p$ is a ``good" prime for which $q$ has order $1$ modulo $p$.
This is based on the work by D. Rossi \cite{Rossi}, where it is proved that for such a good prime $p$, the $p$-subgroup poset of the finite group $G(\GF{q})$ has the homotopy type of the poset of $\GF{q}$-Levi subgroups, which coincides with the poset $\D(\Delta(G,\GF{q}))$.
Here $\Delta(G,\GF{q})$ is the building of $G(\GF{q})$.

We will write $|S|$ to denote the geometric realization of the Levi sphere $S$, which is a sphere of dimension $\dim S$ (see Lemma \ref{lm:convexIsleviSphere}).

The map $F$ naturally extends to an order-preserving map $F:\OPD(\Delta)\to \PD(\Delta)$ as the identity on the $\Delta$-part.
Studying fibers of this map, we prove:

\begin{theorem}
Let $\Delta$ be a spherical building.
Then we have a homotopy equivalence
\begin{equation}
\OPD(\Delta) \simeq \PD(\Delta) \bigvee_{ T \in {\D(\Delta)}} |T| * |T|*\PD(\Lk_{\Delta}(\sigma_T)),
\end{equation}
where $\sigma_T\in T$ is some maximal simplex.
In particular, $\PD(\Delta)$ is spherical of dimension $2\dim\Delta + 1$, and $\CB(\Delta)$ has the homotopy type of a wedge of spheres of dimension $2\dim\Delta + 1$.

Moreover, for a Levi sphere $S$ and a pair of opposite maximal simplices $\sigma_1,\sigma_2\in S$, we have
\begin{align*}
\PD(\Delta)_{\prec S} & \simeq |S| * \PD(\Lk_{\Delta}(\sigma_1))\\
\OPD(\Delta)_{\prec (\sigma_1,\sigma_2)}  & \simeq |S| * \OPD(\Lk_{\Delta}(\sigma_1)).    
\end{align*}
\end{theorem}

Here we are using the bar notation for Levi spheres specifically to distinguish between $S$ as a sphere (via $|S|$) and $S$ as an element of the poset $\D(\Delta)$.

If $\Delta$ is the building of a finite-dimensional vector space $V$, then $\CB(V) = \CB(\Delta)$, $\PD(\Delta)\simeq \PD(V)$ and $\OPD(\Delta)\simeq \OPD(V)$.
In particular, our results recover the connectivity result --- i.e., Rognes's conjecture --- for the common bases complex of a vector space, proved first in \cite{MPW}.
Our proof is independent of that of \cite{MPW}, and it is based on geometric properties of buildings.
To prove our homotopy equivalences, we use the notion of convexity and complete reducibility introduced by Serre.
Key ingredients in our proofs are statements of the form ``a certain (convex) subcomplex $K$ of the building is contractible."  For that, we show that $K$ is not completely reducible, that is, there is a simplex in $K$ without an opposite in $K$.
This implies that $K$ is contractible by \cite{Serre}.
We extend this idea to prove that certain fixed-point subposets are contractible.

\subsection*{Organization of the article}

The paper is organized as follows.
In Section \ref{sec:preliminaries} we introduce the notation from algebra, combinatorics, geometry, and topology that we will use throughout the paper. We also recall results that will play an important role in the rest of the paper. In particular,  we outline the results of the theory of Tits buildings, mostly following \cite{AB} and sometimes \cite{Tits1}.
In Section \ref{sec:vectorspaces}, we study the motivating type A situation.
More precisely, we give exact definitions of the ordered and unordered (partial) decomposition posets and the common basis complex for vector spaces. We also provide the proof that the ordered partial vector space decompositions are equivariantly homotopy equivalent to the poset of ordered partial decompositions of the corresponding building.

In Sections \ref{sec:CBBuildings}, \ref{sec:leviSpheres} and \ref{sec:opd}, we define common basis complexes, ordered and unordered decomposition, and partial decomposition posets for arbitrary buildings.
We provide results about homotopy types and homotopy equivalences, including equivariant versions and fixed points. Some of them extend
facts from the linear case discussed in
Section \ref{sec:vectorspaces}.

Finally, in Section \ref{sec:algebraicgroups} we specialize these results to the case of rational points of connected reductive algebraic groups, obtaining explicit descriptions of the posets and complexes in terms of rational parabolics and Levi subgroups.
In particular, as an application of the equivariance of our maps and the connectivity results, in Proposition \ref{prop:lowLeviIntervalPD} we provide a long exact sequence in terms of Steinberg-square modules of rational Levi subgroups for the rational points of a connected reductive algebraic group, ending in the Steinberg module of such a group.

Computer calculations were performed with GAP \cite{GAP4}
and software package \cite{posets}.

\subsection*{Acknowledgments}
We thank Bernhard Mühlherr and Richard Weiss for helpful and motivating conversations.
The first author was supported by the FWO grant 12K1223N.

\section{Preliminaries}
\label{sec:preliminaries}

The aim of this section is to set the notation on posets and simplicial complexes, and recall some basic facts on the theory of buildings developed by J. Tits.
For the building theory, we will mainly use \cite{AB} and \cite{Tits1}.

Let $X$ be a topological space, and let $m\in \ZZ$.
If $m\leq -2$, then every space is termed $m$-connected.
For $m=-1$, $X$ is $(-1)$-connected if and only if it is non-empty.
If $m\geq 0$, we say that $X$ is $m$-connected if its homotopy groups of degree $\leq m$ vanish (regardless of the base point).
If $X$ is $m$-connected and $Y$ is $n$-connected, then the topological join $X*Y$ is $(m+n+2)$-connected (see \cite{Milnor}).
For our applications, it is
convenient to deviate from the usual definition of join and define the
join of $X$ with the empty space as $X$. Note that the formula for
the connectivity of the join is consistent with this convention if
we consider the empty space as $(-2)$-connected.
For $m\geq 0$, a continuous map $f:X\to Y$ is an $m$-equivalence if it induces an isomorphism between homotopy groups of degree $< m$, and an epimorphism in degree $m$.

We will always work with reduced homology $\widetilde{H}_*(X,R)$ with coefficients in a ring $R$. 
We write $\widetilde{H}_*(X)$ for reduced homology
$\widetilde{H}_*(X,\ZZ)$ with coefficients in the integers.

\subsection{Simplicial complexes}
Let $K$ be a simplicial complex, and write $\dim K$ for its dimension.
By convention, we assume that every (abstract) simplicial complex contains the
empty set. A subset $L \subseteq K$ which is also a simplicial complex is a subcomplex of $K$ and $L$ is called a full subcomplex if it contains all simplices of $K$ contained in the vertex set of $L$. 
When $\dim K$ is finite, we say that $K$ is pure if all of its maximal simplices have dimension $\dim K$.
The geometric realization of $K$ is denoted by $|K|$.
We write $K'$ for the first barycentric subdivision of $K$, so $|K'| = |K|$.
More generally, if $L$ is any set of simplices of $K$, $L'$ denotes the simplicial complex whose simplices are chains on elements of $L$.
If $G$ is a group acting on $K$ by simplicial automorphisms, we say that $K$ is a $G$-complex.
It is clear then that $|K|$ is a $G$-space, and we can talk about its $G$-equivariant homotopy type.
We will also regard $K$ as a poset via its face poset, which we will usually denote by $K$, or $\X(K)$ to stress the difference.
By convention, we think of the empty set as a simplex of $K$, but when working with the face poset of $K$, we remove this element (otherwise the empty simplex would always be a unique minimal element for $\X(K)$, rendering $|\X(K)|$ contractible).
If $\sigma \in K$, then we write $\Lk_{K}(\sigma)$ for the set of simplices of $K$ that contain $\sigma$.
Although $\Lk_{K}(\sigma)$ is not the classical simplicial link, it is naturally isomorphic to it via the map $\tau\mapsto \tau\setminus \sigma$ (so $\sigma$, as the bottom element of $\Lk_{K}(\sigma)$, is regarded as the empty simplex of the classical simplicial link).
Via this identification we can regards $\Lk_K(\sigma)$ as a simplicial
complex.
For the purposes of this article, it will be convenient to work with this definition of the link.
If $K$ is pure, the codimension of a non-empty simplex $\sigma$ is $\codim_K(\sigma) = \dim K - \dim \sigma$, which coincides with $1+\dim \Lk_{K}(\sigma)$.

Suppose now that $K$ is finite-dimensional.
We say that $K$ is spherical if it has the homotopy type of a wedge of spheres of dimension $\dim K$.
Equivalently, $K$ is $(\dim K -1)$-connected.
We say that $K$ is Cohen-Macaulay if $\Lk_K(\sigma)$, regarded as a simplicial complex, is spherical of dimension $\dim K - |\sigma|$ for all $\sigma\in K$.
In particular, if $K$ is Cohen-Macaulay, it is a pure and spherical simplicial complex.
If $\tau$ is a non-empty face of $\sigma$, then the codimension of $\tau$ in $\sigma$ is $\codim_{\sigma}(\tau) = \dim \sigma - \dim\tau = |\sigma|-|\tau|$.

If $K$ and $L$ are $G$-complexes whose geometric realizations have the same ($G$-equivariant) homotopy type, then we write $K\simeq L$ (resp. $K\simeq_G L$).
The join of $K$ and $L$ is the simplicial complex $K * L$ whose underlying vertex set is the disjoint union of the vertex sets of $K$ and $L$, and simplices are of the form $\sigma\cup \tau$, where $\sigma\in K$ and $\tau \in L$.
Then $|K * L|$ is naturally homeomorphic to the topological join $|K|*|L|$.
Note that $K*L$ is naturally a $G$-complex.

Finally, we write $K^G$ for the set of simplices $\sigma \in K$ that are stable under the action of $G$.
This does not mean that $G$ fixes every vertex of $\sigma$, but rather that $G$ permutes these vertices.
Therefore in general, $K^G$ is not a simplicial complex, but $(K^G)' = (K')^G$ is, and $|(K')^G| = |K|^G$.
In particular, we write $|K^G|$ for the geometric realization of $(K^G)'$.
The $G$-stabilizer of a simplex $\sigma \in K$ is denoted by $\Stab_G(\sigma)$.

\subsubsection{Chamber complexes and convexity}
\label{subsub:chambercomplexes}
Suppose that $K$ is a pure simplicial complex.
We denote by $\cham K$ the set of maximal simplices of $K$, which we also call chambers.
From $\cham K$ we can form a graph with vertex set $\cham K$ in which two chambers are adjacent if they share a codimension-one face.
A gallery in $K$ is a path in this graph, that is, a sequence $\sigma_0,\sigma_1,\ldots,\sigma_r$ of chambers of $K$ such that $\sigma_i$ and $\sigma_{i+1}$ share a codimension-one face for all $0\leq i <r$.
In this case, the length of the gallery is $r$.
We assume that the given graph is connected, in which case $K$ is called a chamber complex.
The distance between two chambers of $K$ is the minimal length of a gallery joining them.
The distance between two arbitrary simplices of $K$ is the minimal possible distance between two chambers that respectively contain them.
The diameter of $K$ is the maximal possible distance between two chambers of $K$.
A set $\C$ of chambers of $K$ is convex if every minimal gallery whose extremities lie in $\C$ is completely contained in $\C$.
A convex chamber subcomplex of $K$ is a subcomplex $L$ that is a chamber complex with $\cham L\subseteq \cham K$, such that $\cham L$ is convex as a subset of $\cham K$.
An arbitrary subcomplex $L$ of $K$ is convex if it is an intersection of convex chamber subcomplexes.
Note that convex subcomplexes are always full subcomplexes.
The full convex hull of a set $S$ of simplices of $K$ is denoted by $\Conv_K(S)$ and is the smallest convex subcomplex of $K$ containing $S$.

Finally, simplicial automorphisms preserve convexity, so for any automorphism $\phi$ of $K$, and $S \subseteq K$, we have $\varphi(\Conv_K(S)) = \Conv_K(\varphi(S))$.

\subsection{Posets}
Let $X$ be a poset.
If $x\in X$, then we write $X_{\leq x}$ for the set of elements $y\in X$ such that $y\leq x$.
Analogously defined are $X_{<x}, X_{\geq x}$ and $X_{>x}$.
If $X$ has a unique maximal element $1_X$, we write $\redm{X}$ for the subposet $X\setminus \{1_X\}$.
If in addition $X$ has a unique minimal element $0_X$, then we say that $X$ is bounded and write $\red{X} = X \setminus \{1_X, 0_X\}$ to denote the proper part of $X$.
The order complex of $X$ will be denoted by $\K(X)$.
By the geometric realization of $X$ we mean the geometric realization of its order complex, and we denote it by $|X|$ or $|\K(X)|$.
The dimension (or height) of $X$, denoted by $\dim X$ (resp. $h(X)$), is the dimension of its order complex.
The height of an element $x\in X$ is the dimension of the lower interval $X_{\leq x}$, and we denote it by $h(x)$, or $h_X(x)$ to emphasize the role of $X$.
We write $X^{(i)}$ for the subposet of $X$ consisting of elements of height at most $i$.
The first barycentric subdivision of $X$, denoted by $X'$, is the poset of finite non-empty chains in $X$.
That is, $X' = \X(\K(X))$.
Finally, we write $X^{\op}$ for the opposite poset of $X$.  So, $X^{op}$ is obtained from $X$ by reversing the ordering relation.
Note $\K(X)$ is the same as $\K(X^{\op})$, hence $X' = (X^{\op})'$, and $|X| = |X^{\op}|$.

If $G$ is a group, we say that $X$ is a $G$-poset if it admits an action of $G$ by poset automorphisms.
From now on, we assume that $X$ is a $G$-poset.
Thus, $\K(X)$ is a $G$-complex, and the ($G$-equivariant) homotopy type of $X$ is the ($G$-equivariant) homotopy type of its geometric realization.

We say that $X$ is $m$-connected for $m\in \ZZ$ if its order complex is.
Similarly, if $X$ is finite-dimensional, we say that $X$ is spherical (resp. Cohen-Macaulay) if its order complex is.

If $X$ and $Y$ are two posets, their join, denoted by $X\ojoin Y$, is the poset whose underlying set is the disjoint union of $X$ and $Y$, and the ordering is as follows.
We keep the same ordering among elements of $X$ and among elements of $Y$, and every element of $X$ is less than every element of $Y$.
Note that $X\ojoin Y$ and $Y\ojoin X$ are not isomorphic in general.
Nevertheless, it is evident that $\K(X\ojoin Y) \cong \K(X) * \K(Y) \cong \K(Y\ojoin X)$, so $|X\ojoin Y|\cong |X| * |Y|$.
When there is no risk of confusion --- for instance, when talk about the homotopy type --- we will use the classical join notation and write $X * Y$ for $X\ojoin Y$.

If $f,g:X\to Y$ are two order-preserving $G$-equivariant maps between $G$-posets, then we write $f\leq g$ if $f(x)\leq g(x)$ for all $x\in X$.
In such a case, $f,g$ induce $G$-equivariant homotopic maps on the geometric realizations.
In particular, if $X$ has a unique minimal element, or a unique maximal element, then $X$ is $G$-contractible.
If $X$ and $Y$ have the same ($G$-equivariant) homotopy type, we write $X\simeq Y$ (resp. $X\simeq_G Y$).
If $f:X \to X$ is a $G$-equivariant and order-preserving map with $f(x) \leq x$ for all $x \in X$, then $f$ induces a $G$-equivariant homotopy equivalence $X \simeq_G \Im(f)$.

We denote by $X^G$ the subposet of $X$ whose elements are fixed by $G$.
Note that $\K(X^G) = \K(X)^G$ is a simplicial complex, and $|X^G| = |X|^G$.
Similarly, we write $\Stab_G(x)$ for the $G$-stabilizer of a point $x\in X$.

We will use the following theorems to compute the homotopy type of posets and prove that certain maps are homotopy equivalences.

\begin{theorem}
[Quillen's fiber theorem \cite{Qui78}]
\label{thm:quillen}
Let $f:X\to Y$ be an order-preserving map between posets.
If $f^{-1}(Y_{\leq y})$ is contractible for all $y\in Y$, then $f$ is a homotopy equivalence.
\end{theorem}

Another useful tool that we will use is the wedge decomposition:

\begin{theorem}
\label{thm:wedgeDecomposition}
Let $X,Y$ be posets of finite dimension.
Suppose that $f:X\to Y$ is an order-preserving map such that for all $y\in Y$, $f^{-1}\big(\,Y_{<y}\,\big) \hookrightarrow f^{-1}\big(\,Y_{\leq y}\,\big)$ is homotopic 
to a constant map with image $c_y\in f^{-1}(Y_{\leq y})$.

Then there is a homotopy equivalence
\[ X \simeq Y \vee \bigvee_{y\in Y} f^{-1}\big(\,Y_{\leq y}\,\big) * Y_{>y},\]
where the wedge identifies $y\in Y$ with $c_y\in f^{-1}(Y_{\leq y})$.
\end{theorem}

The statement of Theorem \ref{thm:wedgeDecomposition} on its full generality is given in \cite[Proposition 5.1]{PW25}.
Note that if $y\in Y$ is minimal, $f^{-1}(Y_{< y})$ is empty, and we regard the inclusion of the empty set $\emptyset \hookrightarrow f^{-1}(Y_{\leq y})$ as null-homotopic if $f^{-1}(Y_{\leq y})$ is non-empty (cf. \cite[Theorem 2.5]{fibertheorems}).

For equivariant homotopy types, we will use the Whitehead equivariant theorem.

\begin{theorem}
[Equivariant Whitehead]
\label{thm:whitehead}
Let $G$ be a compact Lie group.
Suppose that $f:X\to Y$ is a $G$-equivariant continuous map between $G$-CW-complexes such that, for all $H\leq G$, the induced map $f_H:X^H\to Y^H$ is a homotopy equivalence.
Then $f$ is a $G$-equivariant homotopy equivalence.
\end{theorem}

\begin{proof}
See \cite{GreenleesMay, Matumoto}.
\end{proof}

We are not aware of an equivariant version of Theorem \ref{thm:wedgeDecomposition}.
Such a result would require subtle assumptions on the choice of gluing points. 
Nevertheless, we can use Theorem 9.1 from \cite{fibertheorems} in the context of finite-dimensional posets and finite groups to yield a homology-equivariant version of Theorem \ref{thm:wedgeDecomposition}.

\begin{theorem}
\label{thm:homologyWedge}
Suppose that $G$ is a finite group, and $f:X\to Y$ is a $G$-equivariant order-preserving map between finite-dimensional $G$-posets.
Let $R$ be a ring such that $RG$ is semisimple.
Then, under the conditions of Theorem \ref{thm:wedgeDecomposition}, for all $m\geq 0$ we have an isomorphism of $RG$-modules:
\[ \widetilde{H}_{m}(X, R) \cong_G \widetilde{Y}_m(Y, R) \oplus \bigoplus_{\overline{y} \in Y/G} \bigoplus_{i+j=m-1} \Ind_{\Stab_G(y)}^G \big( \widetilde{H}_i(f^{-1}(Y_{\leq y})) \otimes \widetilde{H}_j(Y_{> y})   \big). \]
\end{theorem}

Here we are denoting by $\overline{y}$ the image of an element $y\in Y$ in the orbit poset $Y/G$.

\subsection{Buildings}
We work with spherical buildings in the sense of \cite{AB}, so we do not assume that our buildings are thick.
From now on, $\Delta$ will denote a (spherical) building.
By the Solomon-Tits theorem, $\Delta$ is spherical in the sense that it has the homotopy type of a wedge of spheres of dimension $\dim \Delta$ \cite[Theorem 4.73]{AB}.
On the other hand, for every simplex $\sigma\in \Delta$, its link $\Lk_{\Delta}(\sigma)$ is a building of dimension $\codim_{\Delta}(\sigma)$ \cite[Proposition 4.9]{AB}.
In particular, buildings are Cohen-Macaulay.

We write $\A(\Delta)$ for the complete system of apartments of $\Delta$.
From now on, when we speak of an apartment of $\Delta$ we mean an
element of $\A(\Delta)$. 
If $\sigma_1,\ldots,\sigma_r\in \Delta$ are simplices, we denote by $\A(\Delta,\sigma_1,\ldots,\sigma_r)$ the set of apartments $\Sigma\in \A(\Delta)$ such that $\sigma_i\in \Sigma$ for all $i$.
Recall that, since we work with spherical buildings, apartments are finite Coxeter complexes.

\subsubsection{Opposition}
\label{sub:opposition}
Let $\Sigma$ be a (finite) Coxeter complex.
Two chambers of $\Sigma$ are called opposite if their distance coincides with the diameter of $\Sigma$.
It is well known that every chamber of $\Sigma$ has a unique opposite, giving rise to an involutory bijection $\cham \Sigma\to \cham \Sigma$ that extends uniquely to an involutory automorphism $\op_{\Sigma}:\Sigma\to\Sigma$ \cite[2.39]{Tits1}.
Thus, two simplices $\sigma,\sigma' \in \Sigma$ are opposite if $\op_{\Sigma}(\sigma) = \sigma'$.
By convention, the opposite of the empty simplex is the empty simplex.

Two simplices of $\Delta$ are called opposite if they are opposite in some apartment (and hence in every apartment that contains both of them).
It follows that two opposite chambers lie in a unique apartment (see \cite[Lemma 4.69]{AB} and \cite[3.25]{Tits1}).
Note that a simplex may have multiple opposites in $\Delta$.

There is a bijection between apartments containing two given opposite simplices and apartments in the link of one of these simplices:

\begin{lemma}
\label{lm:bijectionApartmentsLinkOpposite}
Let $\Delta$ be a spherical building, and let $\sigma,\sigma'$ be two opposite simplices.
Then we have a bijection:
\[ \Sigma\in \A(\Delta, \sigma,\sigma') \longmapsto \Lk_{\Sigma}(\sigma)\in \A(\Lk_{\Delta}(\sigma)).\]
The inverse of this map is given as follows.
If $\widetilde{\Sigma}\in \A(\Lk_{\Delta}(\sigma))$ and $c,c'\in \widetilde{\Sigma}$ are two opposite chambers there, then $\widetilde{\Sigma} = \Lk_{\Sigma}(\sigma)$ where $\Sigma$ is the convex hull in $\Delta$ of the opposite chambers $c$ and $\proj_{\sigma'}(c')$. Thus we map $\widetilde{\Sigma}$ to $\Sigma$.
\end{lemma}

\begin{proof}
This is part of Proposition 2.1 in \cite{vH}.
\end{proof}

Here, $\proj_{\sigma}(\tau)$ denotes the projection of a simplex $\tau$ to $\sigma$, in the sense of Tits \cite[2.30]{Tits1} (see next subsection for the definition).

\subsubsection{Convexity}
\label{sub:convexity}
Now we look at convex subcomplexes of $\Delta$, in the sense we defined above.
Our notion of convexity coincides with Tits' definition \cite[1.5]{Tits1}, and it is stronger than the notion given in \cite[Definition 4.120]{AB}.
See also \cite[Remark 4.122]{AB}.
Moreover, we will mostly work with convex subcomplexes of $\Delta$ that are contained in some apartment (from the complete system of apartments).

Suppose first that $\Sigma$ is a Coxeter complex.
A root of $\Sigma$ is the image of a folding \cite[1.8]{Tits1}.
If $\alpha$ is a root of $\Sigma$, we write $-\alpha$ for its opposite root, and $\partial \alpha =\alpha \cap (-\alpha)$ is (by definition) a wall (see discussion at the top of page 11 in \cite{Tits1}).
Note that $-\alpha = \op_{\Sigma}(\alpha)$.
Write $\roots(\Sigma)$ for the set of roots of $\Sigma$.
Recall that a panel is a codimension-one simplex.
Given a chamber $\sigma\in \Sigma$ and a panel $\tau\subseteq \sigma$, there exists a unique root $\alpha\in \roots(\Sigma)$ such that $\tau\in \partial \alpha$ and $\sigma\in \alpha$.
Every root (and hence every wall) is a convex subcomplex of $\Sigma$.
Indeed, by \cite[2.19]{Tits1}, convex subcomplexes of $\Sigma$ are exactly those obtained as intersections of roots.
If $K$ is a subcomplex of $\Sigma$, we write $\roots_{\Sigma}(K)$ for the set of roots $\alpha\in \roots(\Sigma)$ that contain $K$.
Thus, $K$ is convex if and only if
\[ K = \bigcap_{\alpha\in \roots_{\Sigma}(K)} \alpha.\]

The convex hull of two opposite simplices $\sigma,\sigma'\in \Sigma$ is termed a Levi sphere of the Coxeter complex $\Sigma$.
Indeed, Levi spheres are exactly the subcomplexes of $\Sigma$ obtained as intersections of walls.
This terminology is
borrowed from the work in \cite{Serre} of Serre, who defines a Levi sphere as an intersection of $|\Sigma|$, viewed as the classical Euclidean sphere, with sets of reflecting hyperplanes of the underlying Coxeter group.
This concept generalizes the notion of Levi subgroups in algebraic groups to the context of buildings, as we will explain later in Lemma \ref{lm:LeviSubgroupsAndLeviSpheres}.

Now we go back to $\Delta$.
Let $S$ be a set of simplices of $\Delta$.
If $S \subseteq \Sigma$ for some apartment $\Sigma$, then the convex hull of $S$ is the same whether it is taken in $\Sigma$ or in $\Delta$, that is, $\Conv_{\Sigma}(S) = \Conv_{\Delta}(S)$.
If $\sigma,\tau\in \Delta$ are two simplices, then they lie in some apartment $\Sigma$.
The projection of $\tau$ onto $\sigma$, denoted by $\proj_{\sigma}(\tau)$ is the unique maximal simplex containing $\sigma$ in the convex hull $\Conv_{\Delta}(\sigma,\tau)=\Conv_{\Sigma}(\sigma,\tau)$ (see \cite[2.30, 3.19]{Tits1}).

A Levi sphere of $\Delta$ is the convex hull of two opposite simplices, which is then a Levi sphere of any apartment containing these simplices.
If $\sigma,\sigma'$ are opposite chambers, then $\Conv_{\Delta}(\sigma,\sigma')$ is the unique apartment that contains them and hence a $\dim \Delta$-sphere.
If $\sigma,\sigma'$ are opposite vertices, then $\Conv_{\Delta}(\sigma,\sigma') = \{\sigma,\sigma'\}$ is a $0$-sphere.

More generally, Levi spheres are always spheres (see also Remarques 2 
on the bottom of page 200 in \cite{Serre}).

\begin{lemma}
\label{lm:convexIsleviSphere}
Let $\Sigma$ be a finite Coxeter complex, and let $K\subseteq \Sigma$ be a convex subcomplex.
Then $K$ is a Levi sphere if and only if $K$ contains a pair of opposite simplices $\sigma_1,\sigma_2$ with dimension $\dim \sigma_i = \dim K$.
In such a case, $K$ is the convex hull of $\sigma_1,\sigma_2$ and $|K|$ is a sphere of dimension $\dim K$.
\end{lemma}

\begin{proof}
Suppose that $K$ contains a pair of opposite simplices $\sigma_1,\sigma_2$ of dimension $\dim K$.
Then, in the terminology of \cite[Théorème 2.1]{Serre}, $K$ must be completely reducible. This means that every point of $K$ has an opposite.
But every point of $K$ has at most one opposite as $K$ lies in the sphere $\Sigma$, so $K$ must be exactly the Levi sphere spanned by $\sigma_1$ and $\sigma_2$.
In particular, $K$ is the convex hull of $\sigma_1,\sigma_2$, and it triangulates a sphere of dimension $\dim K$.

The converse of this statement is clear (see \cite{Serre}).
\end{proof}

A root (resp. a wall) of $\Delta$ is a root (resp. a wall) of some apartment.
We write $\roots(\Delta)$ for the set of roots of $\Delta$, so
\[ \roots(\Delta) = \bigcup_{\Sigma \in \A(\Delta)} \roots(\Sigma).\]
Similarly, if $K\subseteq \Delta$, $\roots_{\Delta}(K)$ denotes the set of roots $\alpha\in \roots(\Delta)$ that contain $K$.
If $K$ is a convex subcomplex of $\Delta$ that is contained in some apartment, then $K$ is a convex subcomplex of any apartment containing it.
In particular, if $K\subseteq \Sigma\in \A(\Delta)$, then 
$K = \bigcap_{\alpha \in \roots_{\Sigma}(K)} \alpha$, and more generally,
\[ K  =\bigcap_{\alpha\in \roots_{\Delta}(K)} \alpha.\]

\subsubsection{The CAT(1) metric}
\label{subsub:cat1metric}
The geometric realization of a spherical building $\Delta$ admits a canonical metric $\dcat$ that makes it a complete CAT(1) space.
See \cite[II.10 Theorem 10A.4]{BH} and \cite[Example 12.39]{AB}.
With this metric, an apartment becomes isometric to the unit sphere $\SS^d$ of the vector space on which the underlying reflection group acts, where  $d = \dim \Delta$.
In particular, the opposition map $\op_{\Sigma}:\Sigma\to\Sigma$ of an apartment $\Sigma$ gives rise to the involution $-\id_{\SS^d}:\SS^d\to \SS^d$ of the unit sphere.
Hence, two simplices $\sigma,\sigma'\in \Sigma$ are opposite if and only if their barycenters are opposite points when regarded in $\SS^d$ via this identification.

The diameter of $|\Delta|$, which is $\pi$ with this metric, is also the diameter of any apartment, which equals the distance between two opposite points.
In particular, for two points $x,y\in |\Delta|$ at distance $\dcat(x,y) < \pi$, there exists a unique geodesic from $x$ to $y$.
Recall that a subspace $X$ of a CAT(1) space is convex if for every two points $x,y\in X$ at distance $<\pi$, the unique geodesic segment joining $x,y$ is completely contained in $X$.
It follows that a convex subcomplex $K$ of $\Delta$ gives rise to a convex subspace of $|\Delta|$.
Notice that not every convex subspace of $|\Delta|$ arises in this way.
Now, if $\sigma,\sigma'$ are two non-opposite simplices in $\Delta$, then there is a unique geodesic (in the geometric realization) that joins their barycenters.

\subsubsection{Automorphism group}
\label{subsub:autDelta}
We denote by $\Aut(\Delta)$ the group of simplicial automorphisms of $\Delta$.
These automorphisms might not be type-preserving.
By a group acting on $\Delta$, we mean a group inducing simplicial automorphisms on $\Delta$.

Any simplicial automorphism on $\Delta$ gives rise to an isometry of $|\Delta|$ with the metric $\dcat$.
Therefore, if $H$ is a group acting simplicially on $\Delta$, and $x,y\in |\Delta|^H$ are two points at distance $<\pi$, then $H$ must fix the unique geodesic joining $x$ and $y$.
In particular, if $\sigma,\sigma'$ are non-opposite simplices of $\Delta$ that are invariant under the action of $H$, then $H$ fixes the unique geodesic in $|\Delta|$ that joins their barycenters.

From these observations, we get the following lemma.

\begin{lemma}
\label{lm:contractibleIntersectionApartments}
Let $\Delta$ be a spherical building, and let $H$ be a group acting on $\Delta$ by simplicial automorphisms.
Let $\tau\in \Delta^H$ be a non-empty simplex fixed by $H$, and let $\S \subseteq \A(\Delta,\tau)$ be a set of apartments containing $\tau$.

If $\bigcap_{\Sigma\in \S} (\Sigma^H)$ does not contain an opposite of $\tau$, then it is contractible.
\end{lemma}

\begin{proof}
Let $X := \bigcap_{\Sigma\in \S} \Sigma^H$, and note that $|X| =  \bigcap_{\Sigma\in \S} |\Sigma'|^H$.
Assume that there is no opposite of $\tau$ in $X$.
Then every point $x\in X$ is at distance $ < \pi$ from the barycenter of $\tau$ (say, $x_0$).
Hence, by the discussion in Subsection \ref{subsub:autDelta}, $H$ fixes the unique geodesic joining $x$ with $x_0$.
Thus, we can contract $X$ to $x_0$ using these geodesics.
\end{proof}

\section{Posets and simplicial complexes in the case of vector spaces}
\label{sec:vectorspaces}

In this section we take a closer look at the poset of partial decompositions, the poset of ordered partial decompositions
and the common bases complex for vector spaces. 
We demonstrate how building-related constructions show up.
The results then motivate and guide the more general definitions for buildings in Sections \ref{sec:CBBuildings}, \ref{sec:leviSpheres} and
\ref{sec:opd}.

Let $V$ be a finite-dimensional vector space defined over a field $k$.
Write $T(V)$ for the poset of proper non-zero subspaces of $V$, ordered by inclusion.
We denote by $\Delta(V)$ the order complex of $T(V)$, which is also the building associated with the group $\GL(V)$.

\begin{definition}
\label{def:DandPDVectorSpaces}
Let $V$ be a finite-dimensional vector space defined over a field $k$.
A partial decomposition of $V$ is a subset $\{S_1,\ldots,S_r\} \subseteq T(V)$ such that
\[ \gen{S_1,\ldots,S_r} \cong S_1\oplus \cdots \oplus S_r.\]

We denote by $\PD(V)$ the poset of partial decompositions of $V$ other than $\emptyset$ and $\{V\}$, with order given by refinement; that is
for $d_1,d_2\in \PD(V)$
\[ d_1\leq d_2 \text{ if for all } S\in d_1 \text{ there is } T\in d_2 \text{ such that } S\leq T.\]

A full decomposition $V$ is a partial
decomposition $\{S_1,\ldots, S_r\}$ such that $V \cong S_1 \oplus\cdots \oplus S_r$. The poset of full decompositions of $V$ is the subposet
$\D(V)$ of $\PD(V)$ on the set of full decompositions of $V$.
\end{definition}

Note that we are not including $\{V\}$ in the poset $\D(V)$
and that 
\[ d \text{ is a partial decomposition } \Leftrightarrow
\ \dim \gen{S\tq S\in d} = \sum_{S\in d} \dim S.\]

We will also work with the ordered versions of $\PD(V)$ and $\D(V)$:

\begin{definition}
\label{def:ODandOPDVectorSpaces}
Let $V$ be a finite-dimensional vector space defined over a field $k$.
The poset of ordered partial decompositions of $V$, denoted by $\OPD(V)$ consists of tuples of distinct subspaces $(S_1,\ldots,S_r)$ such that $\{S_1,\ldots,S_r\}\in \PD(V)$.
The ordering in $\OPD(V)$ is given by refinement that preserves the order of the elements of the tuples.
That is, if $d_1 = (S_1,\ldots,S_r)$ and $d_2 = (W_1,\ldots,W_t)$, then $d_1\leq d_2$ if for all $1\leq i\leq j\leq r$ there are $1\leq k\leq l\leq t$ such that $S_i\leq W_k$ and $S_j\leq W_l$.

The poset $\OD(V)$ of ordered full decompositions of $V$ is the subposet of
$\OPD(V)$ on the set of $(S_1,\ldots,S_r)$ such that $\{S_1,\ldots,S_r\}\in \D(V)$.
\end{definition}

As observed first in \cite{LR}, the poset $\OD(V)$ is naturally isomorphic to the poset of opposite pairs of the building $\Delta(V)$.

\begin{remark}
[{Opposite simplices}]
\label{rk:ODlinearCase}
An ordered full decomposition $d = (S_1,\ldots,S_r)\in \OD(V)$ determines the pair $\big(\,P(d),Q(d)\,\big)$ of flags given by
\[ P(d) = (\,S_1 < S_1 \oplus S_2 < \cdots < S_1\oplus\cdots \oplus S_{r-1}\,),\]
and
\[ Q(d) = (\, S_r < S_r\oplus S_{r-1} < \cdots < S_2\oplus\cdots \oplus S_r\,).\]
In the language of buildings, this means that $P(d),Q(d)$ are opposite simplices of $\Delta(V)$.
For $\Delta(V)$ the notion of opposition
of simplices from Subsection \ref{sub:opposition} translates into
the following. Two simplices $\sigma$ and $\tau$ from $\Delta(V)$ are opposite if they have the same dimension and for all $S\in \sigma$ there exists a unique $T\in \tau$ such that $V = \gen{S,T} \cong S\oplus T$.

Therefore, the poset $\OD(V)$ can be alternatively described as the poset of pairs of opposite simplices of the building $\Delta(V)$, 
where the ordering is given by coordinate-wise reverse inclusion.
The isomorphism is given by $d\mapsto \big(\,P(d), Q(d)\,\big)$.

Another description of $\OD(V)$ is in terms of the Charney poset.
Recall that the Charney poset $\Ch(V)$ consists of pairs $(S,T)$ of proper non-zero subspaces of $V$ such that $S\oplus T = V$.
The ordering in $\Ch(V)$ is given by zig-zag containment:
\[ (S_1,T_1) \leq (S_2,T_2) \  \Leftrightarrow \ S_1\leq S_2 \text{ and } T_1\geq T_2.\]
Then it is not hard to see that
\[ \OD(V) = \X\big(\,\K(\,\Ch(V)\,)\,\big)^{\op}.\]
The Charney poset was introduced by R. Charney \cite{Charney} in the context of free modules over Dedekind domains $R$, where it was used to establish homological stability results for the linear groups $\GL_n(R)$.
\end{remark}

We have seen that $\OD(V)$ has an intrinsic description in terms of building properties.
Our next theorem describes $\D(V)$ in terms of Levi subgroups, hinting at a possible definition of the full decomposition poset for arbitrary buildings that arise from the BN-pair of the $k$-points of a connected reductive algebraic group.

In what follows, suppose that $\overline{k}$ is the algebraic closure of a field $k$.
We say that $L$ is a $k$-Levi subgroup of $\GL_n(k)$ if $L$ is an algebraic group defined over $k$ and it is the Levi complement in a parabolic subgroup of $\GL_n$ that is also defined over $k$.

\begin{theorem}
\label{thm:decompAndLeviVectorSpaces}
Let $\L(\GL_n,k)$ denote the poset of (proper) $k$-Levi subgroups of $\GL_n$ ordered by inclusion.
Then we have a $\GL_n(k)$-equivariant poset isomorphism $\L(\GL_n,k)\groupiso_{\GL_n(k)} \D(k^n)$.
\end{theorem}

\begin{proof}
Any Levi subgroup $L$ of $\GL_n$ is a connected reductive algebraic group, so it decomposes the underlying module $\overline{k}^n$ into a direct sum $\{\overline{S_1},\ldots, \overline{S}_r\}$ of $L$-invariant subspaces.
If $L$ is defined over $k$, $\{\overline{S_1}\cap k^n,\ldots, \overline{S}_r\cap k^n\}$ gives a direct sum decomposition of $k^n$.
This gives a well-defined order-preserving map $\L(\GL_n,k)\to \D(k^n)$.
The inverse of this map is given by mapping a direct sum decomposition to the $\GL_n$-pointwise stabilizer.
That is, for $d\in \D(k^n)$, we define $L = \bigcap_{S\in d} \Stab_{\GL_n(\overline{k})}( S\otimes_k \overline{k} ).$
\end{proof}

Next, we analyze the homotopy type of $\PD(V)$.
For that, it will be useful to recall the notion of a common basis.

\begin{definition}
[{Common basis}]
\label{def:CBVectorSpaces}
Let $V$ be a finite-dimensional vector space.
We say that a collection $\sigma$ of subspaces of $V$ has a common basis if there is a basis $B$ of $V$ such that $B\cap W$ is a basis of $W$ for all $W\in \sigma$.

We let $\CB(V)$ denote the simplicial complex whose simplices are collections of proper and non-zero subspaces of $V$ with a common basis.
\end{definition}

Note that a collection of subspaces with a common basis in a finite-dimensional vector space must be finite.
Thus $\CB(V)$ is indeed a simplicial complex of dimension $2^{\dim V}-3$.

The complex $\CB(V)$ was introduced in \cite{Rognes} by Rognes, who considered its suspension to be a stable building for $V$ and conjectured that indeed $\widetilde{H}_k(\CB(V),\ZZ) = 0$ for $k \neq 2n-3$.
In fact, Rognes defined the common bases complex for finite-rank free modules over local rings and Euclidean domains, and the conjecture is stated in such a more general context.
This conjecture was settled (for vector spaces) by Miller, Patzt and Wilson \cite{MPW}.

In \cite{BPW24}, the authors showed that $\CB(V)$ has the homotopy type of $\PD(V)$, so Rognes's conjecture actually states that $\PD(V)$ is a spherical poset.

\begin{theorem}
[Br\"uck--Piterman--Welker]
\label{thm:PDandCBVectorSpaces}
Let $V$ be an $n$-dimensional vector space, and $H$ a group acting on $V$.
Then there is an $H$-equivariant map $\CB(V) \to \PD(V)$ that is a homotopy equivalence.
\end{theorem}

An earlier unpublished result by Hanlon, Hersh and Shareshian \cite{HHS} shows that $\PD(V)$ is Cohen-Macaulay for finite fields. This result together with Theorem \ref{thm:PDandCBVectorSpaces} provides an independent proof of  Rognes' conjecture for finite fields. Also given in \cite{HHS} is a description of the $GL(V)$-module structure of $\widetilde{H}_{2n-3}(\PD(V),\mathbb{C})$.

In what follows, we provide an alternative ``up-to-homotopy" description of the $\PD(V)$ poset in terms of the building and the poset of decompositions (Levi subgroups).
Recall that $\Delta(V)$ denotes the order complex of the poset of proper non-zero subspaces of $V$, which is the building of $V$.

\begin{definition}
\label{def:PDBuildingVectorSpace}
Let $V$ be a finite-dimensional vector space.
Let $\OKtwo{V}$ (resp., $\Ktwo{V}$) be the poset on 
the (disjoint) union of $\X(\Delta(V))$, and $\OD(V)$ (resp., $\D(V)$).
The order relation $\preceq$ in $\OKtwo{V}$ is the induced order on both
$\Delta(V)$ 
and  $\OD(V)$ 
(resp. $\D(V)$). 
For $\sigma\in\Delta(V)$ and $d\in \OD(V)$ (resp., $d \in \D(V)$) we set $\sigma \prec d$ if the set of subspaces appearing in at least one of  $\sigma$ and $d$ has a common basis.
\end{definition}

It is not hard to see that $\preceq$ is indeed a poset relation for $\OKtwo{V}$ and for $\Ktwo{V}$.

The next theorem says that each of the posets $\OKtwo{V}$ and $\Ktwo{V}$ has the same $\GL(V)$-equivariant homotopy type as the corresponding partial decomposition poset.

\begin{theorem}
\label{thm:PDandPDBuildingVectorSpace}
Let $V$ be an $n$-dimensional vector space over a field $k$.
Let $\sigma \in \OPD(V)'$, and write $\sigma = \sigma_0\cup \sigma_1$ where $\sigma_0$ is the set of ordered partial decompositions in $\sigma$ that span a proper subspace of $V$, and $\sigma_1$ is the set of full ordered decompositions in $\sigma$.
Then, for $\sigma \in \OPD(V)'$, we set
\[ \phi(\sigma) = \begin{cases}
\{ \gen{s} \tq s\in \sigma_0\} & \sigma_1 = \emptyset,\\
\max \sigma_1 & \sigma_1\neq\emptyset.
\end{cases}\]
Then $\phi$ is a $\GL(V)$-equivariant poset map $\phi: \OPD(V)' \to \OKtwo{V}$ such that for all $H\leq \GL(V)$, the induced map $\phi_H: (\OPD(V)')^H \to \OKtwo{V}^H$ is a homotopy equivalence.

We get analogous conclusions for the corresponding map $\PD(V)'\to \Ktwo{V}$.
\end{theorem}

\begin{proof}
We prove the theorem for the ordered version.
The unordered version follows from similar arguments.
Let $H$ be any subgroup of $\GL(V)$, and let $X = (\OPD(V)')^H$, and $Y = \OKtwo{V}^H$.
Clearly, $\phi_H:X\to Y$ gives a well-defined order-preserving map.
Note that $(\OPD(V)')^H = (\OPD(V)^H)'$.

Now we invoke Quillen's fiber theorem to prove that $\phi_H$ is a homotopy equivalence.
So, we fix $y\in Y$ and show that $\phi_H^{-1}(Y_{\preceq y})$ is contractible.
We split into cases according to $y = \tau \in \Delta(V)^H$ or $y = d\in \OD(V)^H$.
Observe that if $\sigma \in X$, then every partial decomposition and subspace involved in $\sigma$ must be fixed by $H$.
Similarly, if $y\in Y$, then every chain and subspace involved in $y$ must be fixed by $H$.

\bigskip

\noindent
\textbf{Case 1.} $y = \tau\in \Delta(V)^H$.

Here we have
\[  \phi^{-1}_H(Y_{\preceq y}) = \phi^{-1}_H(\Delta(V)_{\subseteq \tau}) = \big( \{ s \in \OPD(V)^H  \tq \gen{s} \in \tau\} \big)'.\]
Let $W :=  \{ s \in \OPD(V)^H  \tq \gen{s} \in \tau\}$.
Then, if $S$ denotes the maximal element of $\tau$ (which must be fixed by $H$), we have that $s\leq (\gen{s}) \leq (S)$ for all $s\in Z$.
As $(S)\in W$, $W$, and hence $W' = \phi^{-1}_H(Y_{\preceq y})$, are contractible.

\bigskip

\noindent
\textbf{Case 2.} $y = d\in \OD(V)^H$.

This case is slightly more delicate.
First, we note that if $\sigma = \sigma_0\cup \sigma_1\in \phi^{-1}_H(Y_{\preceq y})$ and $\sigma_1\neq \emptyset$, then by transitivity we have that $\sigma \cup \{d\}\in (\OPD(V)^H)'$ and hence lies in the preimage $\phi^{-1}_H(Y_{\preceq y})$.
For $\sigma \in \phi^{-1}_H(Y_{\preceq y})$, we define
\[ h_1(\sigma) = \begin{cases}
    \sigma & \sigma_1 = \emptyset,\\
    \sigma \cup \{d\} & \sigma_1\neq\emptyset.
\end{cases}\]
Then, $h_1$ is an upward-closed operator, giving rise to a homotopy equivalence $\phi^{-1}_H(Y_{\preceq y}) \simeq \Im(h_1)$.
Next, as the image of $h_1$ contains those $\sigma\in \phi^{-1}_H(Y_{\preceq y})$ that either have $\sigma_1=\emptyset$ or $d\in \sigma_1$, we can perform one more homotopy and only keep $d$ in the $\sigma_1$-part.
That is, for $\sigma \in \Im(h_1)$, let
\[h_2(\sigma) = \begin{cases}
    \sigma & \sigma_1 = \emptyset,\\
    \sigma_0 \cup \{d\} & \sigma_1\neq\emptyset.
\end{cases}\]
Then we have that $h_2(\sigma) \subseteq \sigma$, giving rise to a new homotopy equivalence $\Im(h_1) \simeq \Im(h_2)$,
where
\begin{align*}
    \Im(h_2) & = \{ \sigma\in (\OPD(V)^H)' \tq \sigma = \sigma_0, \phi(\sigma_0) \preceq d\} \cup \{ \sigma\in (\OPD(V)^H)' \tq \sigma_1 = \{d\}\}\\
    & = \big( \{s\in \OPD(V)^H \setminus \OD(V) \tq ( \gen{s}) \preceq d\} \cup \{d\} \big)'\\
    & = \big( U \cup F \big)',
\end{align*}
where
\begin{align*}
    U & := \{d\} \cup \{ s\in \OPD(V)^H \setminus \OD(V) \tq s\leq d\},\\
    F & := \{s\in \OPD(V)^H \setminus \OD(V) \tq ( \gen{s} ) \preceq d, \, s\nleq d\}.
\end{align*}

Finally, we show that 
\[ Z :=  U \cup F \]
is contractible.
Note that $F$ is an upward closed subset of $Z$.
Let $F_2 = \{ (\gen{s}) \tq s\in F\}$.
Then, for $s\in Z$ we define
\[ h_3(s) = \begin{cases}
    s & s \in U,\\
    (\gen{s}) & s\in F.
\end{cases}\]
Since $h_3$ is a well-defined order-preserving map with $h_3(s) \leq s$ for all $s\in Z$, it defines a homotopy equivalence between $Z$ and its image $\Im(h_3) = U \cup F_2$.
Next, write $d=(W_1,\ldots,W_r)$, and if $s\in \OPD(V)^H$ and $\gen{s}, d$ have a common basis, then we can write
\[\psi(s) = (\gen{s}\cap W_1,\ldots,\gen{s}\cap W_r) \in \OPD(V)^H,\]
where we remove those terms that give a null intersection.
Note that, since $\gen{s},d$ have a common basis, $\psi(s) \leq d$ in $\OPD(V)^H$ and
\[ (\gen{s}) = (\gen{\psi(s)}) \geq \psi(s).\]
Moreover, if $s\in U$, then $\psi(s) = s$.
Therefore, for any $s\in \Im(h_3)$ we have
\[ s\geq \psi(s) \leq d.\]
This proves that $Z\simeq \Im(h_3)$ is contractible, finishing the proof.
\end{proof}

\section{The common basis complex for buildings}
\label{sec:CBBuildings}

In this section, we define common basis complexes for spherical buildings, extending the original definition for the linear case (see Definition \ref{def:CBVectorSpaces}).

\begin{definition}
\label{def:commonBasis}
Let $\Delta$ be a spherical building.
We write $\CB(\Delta)$ for the simplicial complex on the same
vertex set as $\Delta$ whose maximal simplices are 
$\displaystyle{\bigcup_{\sigma \in \Sigma}} \sigma$ for 
$\Sigma \in \A(\Delta)$. 

\end{definition}

In other words, if $\Sigma$ is an apartment in the complete system of apartments $\A(\Delta)$ of $\Delta$, and $V(\Sigma)$ denotes the vertex set of $\Sigma$, then $V(\Sigma)$ is a maximal simplex of $\CB(\Delta)$.
Conversely, every maximal simplex of $\CB(\Delta)$ is of that form.

If $K$ is a simplicial complex, full subcomplexes of $K$ are in one-to-one correspondence with subsets of the vertex set of $K$.
Therefore, we can identify a simplex of $\CB(\Delta)$ with the full subcomplex of $\Delta$ induced by $K$ on the set of vertices of the simplex.
This view point allows us to consider $\CB(\Delta)$ as the poset 
whose elements are full subcomplexes contained in some apartment 
from $\A(\Delta)$. 
Thus the maximal elements of $\CB(\Delta)$ are
the apartments of $\Delta$.
From now on, we will adopt this point of view.

\begin{example}
Let $\Delta$ be the building arising as the order complex of the poset of proper non-zero subspaces of a finite-dimensional vector space $V$.

For every apartment $\Sigma$ in $\A(\Delta)$ there is 
a basis of $V$ such that $\Sigma$ is the order complex of the proper
part of the poset of all subspaces spanned by the basis. Each poset
is a Boolean lattice on the basis.

Hence, an element of $\CB(\Delta)$ is a set of non-zero proper subspaces of $V$ that have a common basis.
This is exactly the common bases complex as described by Rognes \cite{Rognes} in the case of vector spaces.
See Definition \ref{def:CBVectorSpaces}.
\end{example}

Note that $\dim \CB(\Delta)$ is the number of vertices of an apartment minus one.
This number is much larger than $\dim \Delta$.
Indeed, for the linear case, $\dim \Delta = \dim V -2$ and
$\dim \CB(V) = 2^{\dim V} -3$. 

Next, we show up to homotopy equivalence the dimension can be reduced. 

If $K$ is a (full) subcomplex of some apartment $\Sigma$ of $\Delta$, then its convex hull $\Conv_{\Delta}(K) = \Conv_{\Sigma}(K)$ is also a full subcomplex of $\Sigma$.
This defines an upward-closed operator on the face poset of $\CB(\Delta)$, whose image is the poset of non-empty convex subcomplexes of $\Delta$ that lie in some apartment.

\begin{definition}
\label{def:convexSubcomplexes}
Let $\Delta$ be a spherical building. We define $Y(\Delta)$ as the poset of non-empty convex subcomplexes of $\Delta$ that lie in some apartment.
The ordering on $Y(\Delta)$ is given by inclusion.
\end{definition}

We immediately have:

\begin{lemma}
\label{lm:CBequivYDelta}
Let $\Delta$ be a spherical building.
Let $\Conv_{\Delta}:\X(\CB(\Delta)) \to Y(\Delta)$ be the map that takes a subcomplex $K\in \X(\CB(\Delta))$ to its convex hull:
\[ K \mapsto \Conv_{\Delta}(K).\]
Then $\Conv_{\Delta}$ is an $\Aut(\Delta)$-equivariant homotopy equivalence.
\end{lemma}

\begin{proof}
Let $i:Y(\Delta)\hookrightarrow \X(\CB(\Delta))$ denote the inclusion map.
Then $\Conv_{\Delta} \circ i = \id_{Y(\Delta)}$ and
\[ (i\circ \Conv_{\Delta})(K) = \Conv_{\Delta}(K) \supseteq K,\]
that is, $i\circ \Conv_{\Delta} \geq \id_{\X(\CB(\Delta))}$.
As $i$ and $\Conv_{\Delta}$ are equivariant maps, we conclude that $\Conv_{\Delta}$ is an $\Aut(\Delta)$-equivariant homotopy equivalence.
\end{proof}

Therefore, the barycentric subdivision of $\CB(\Delta)$ deformation retracts onto a much smaller-dimensional subcomplex $Y(\Delta)$.
Since $\dim Y(\Delta) = \dim Y(\Sigma)$ for any apartment $\Sigma$ of $\Delta$, based on preliminary computations we propose:

\begin{question}
Let $\Sigma$ be a finite Coxeter complex.
Is it true that $$\dim Y(\Sigma) = \dim \Sigma + \# \{ \text{positive roots}\} \,?$$
\end{question}

The following example shows that in the type $A$ case 
$\dim Y(\Sigma)$ is a least a big as the right-hand side.

\begin{example}
Let $\Sigma$ be the type $A_n$ Coxeter complex.
We identify the roots of $\Sigma$ with the roots represented as vectors in the standard definition of a root system. In particular, the positive roots are $e_i - e_j$ for $1 \leq i < j \leq n$. 
We order the roots $e_i-e_j$ by the lexicographic order on the
index pairs $(i,j)$.
Intersecting the hemispheres corresponding to the roots in this order yields a descending sequence of convex subcomplexes of $\Sigma$ terminating in a chamber. 
Since simplices (regarded as subcomplexes) are convex (see \cite[2.26]{Tits1}), we can extend the
descending sequence by a maximal chain of non-empty 
faces of the chamber. Thus we obtain a chain of
$n+\binom{n}{2}$ convex subcomplexes. Showing that the
dimension of $Y(\Sigma)$ is at least $\dim \Sigma + \# \{ \text{positive roots}\}$ in this case.
\end{example}

We will further explore the homotopy type of $Y(\Delta)$ in the next section.
For that, we need the following result on certain upper intervals in $Y(\Delta)$.

\begin{proposition}
\label{prop:intervalSphereYPoset}
Let $\Delta$ be a spherical building, and let $S\in Y(\Delta)$ be a Levi sphere.
Let $\sigma\in S$ be a maximal simplex.
Then we have a homotopy equivalence
\begin{align*}
    A :  Y(\Delta)_{\supsetneq S} & \longrightarrow  Y(\Lk_{\Delta}(\sigma))\\
     K & \longmapsto K\cap \Lk_{\Delta}(\sigma).    
\end{align*}
\end{proposition}

\begin{proof}
Let $\sigma' \in S$ be the unique opposite of $\sigma$ in $S$.
Note that every root containing a convex subcomplex $K\in Y(\Delta)_{\supsetneq S}$, automatically contains $\sigma$ and $\sigma'$.
From now on, we assume that $\sigma$ is not a chamber in $\Delta$, i.e., $S$ is not an apartment (otherwise the result is vacuous).

First, we show that the map
\begin{align*}
R_{\sigma} : \{ \alpha\in \roots(\Delta) \tq \sigma,\sigma'\in \partial \alpha\} & \longrightarrow \roots(\Lk_{\Delta}(\sigma))\\
\alpha & \longmapsto \Lk_{\alpha}(\sigma)
\end{align*}
is surjective.
For subcomplexes or simplices in $\Lk_{\Delta}(\sigma)$, we will use the tilde notation $\widetilde{x}$.

Let $\widetilde{\alpha}\in \roots(\Lk_{\Delta}(\sigma))$, contained in some apartment $\widetilde{\Sigma}\in \A(\Lk_{\Delta}(\sigma))$.
Let $\Sigma\in \A(\Delta,\sigma,\sigma')$ be the unique apartment such that $\Lk_{\Sigma}(\sigma) = \widetilde{\Sigma}$ (see Lemma \ref{lm:bijectionApartmentsLinkOpposite}).
By \cite[Prop. 3.79]{AB}, there is a unique root $\alpha\in \roots(\Sigma)$ that contains $\sigma$ and $\Lk_{\alpha}(\sigma) = \widetilde{\sigma}$.
Moreover, as $\partial \alpha = \op_{\Sigma}(\partial \alpha)$, $\sigma\in \partial\alpha$ and $\sigma'\in \Sigma$, we conclude that $\sigma'\in \partial\alpha\subseteq \alpha$.
This shows that $R_{\sigma}$ is surjective.

Next, suppose that $K \in Y(\Delta)_{\supsetneq S}$.
Let $\alpha\in \roots_{\Delta}(K)$, and let $\Sigma\in \A(\Delta)$ be an apartment for which $\alpha\in \roots(\Sigma)$.
Since $\sigma,\sigma'\in K \subseteq \alpha$, we see that $\sigma,\sigma' \in \alpha\cap -\alpha = \partial \alpha$.
Therefore, every root containing $K$ has $\sigma$ and $\sigma'$ in its boundary.
In particular, as $K$ is an intersection of roots of $\Sigma$ that contain $\sigma,\sigma'$ in their boundaries, we conclude that $A(K) = K\cap \Lk_{\Delta}(\sigma)$ is an intersection of roots of $\Lk_{\Delta}(\sigma)$, so it is a convex subcomplex of $\Lk_{\Sigma}(\sigma)$, and hence of $\Lk_{\Delta}(\sigma)$.
Moreover, as $S\subsetneq K$, $\sigma$ cannot be a maximal simplex of $K$ by Lemma \ref{lm:convexIsleviSphere}, so $K\cap \Lk_{\Delta}(\sigma) \neq \{\sigma\}$.
This proves that $A$ is a well-defined order-preserving map of posets.

Next, we define a homotopy inverse for $A$.
For $\widetilde{K}\in Y(\Lk_{\Delta}(\sigma))$, let
\[ B(\widetilde{K}) = \bigcap_{\alpha \in R_{\sigma}^{-1}\left (\roots_{ \Lk_{\Delta}(\sigma)}(\widetilde{K}) \right)} \alpha.\]
One can show that $B$ is a well-defined order-preserving map whose image strictly contains $S$.
Moreover, $BA(K) \subseteq K$ and $AB(\widetilde{K}) = \widetilde{K}$.
Thus, $A$ is a homotopy equivalence.
Note that it is indeed equivariant by any group of simplicial automorphisms that fixes $\sigma$ and $\sigma'$ (and hence $S$).
\end{proof}

We consider it an interesting question if the maps $A$ and $R_{\sigma}$ from the preceding proof are actually bijections.

The next example describes the poset $Y(\Delta)$ when $\Delta$ is the $A_2$ Coxeter complex.

\begin{example}
Let $\Sigma$ be the Coxeter complex with diagram $A_2$.
That is, the hexagon as displayed in Figure \ref{fig:hexagon}.
Note that $\Sigma$ is a building since we do not require buildings to be thick.

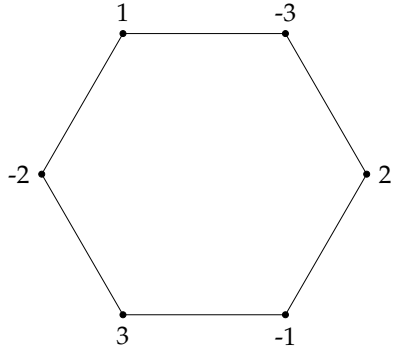
\begin{figure}[hb]
    \centering
    \scalebox{0.8}{
    \begin{tikzpicture}
   \newdimen\R
   \R=2.7cm
   \draw (0:\R) \foreach \x in {60,120,...,360} {  -- (\x:\R) };
   \foreach \x/\l/\p in
     { 60/{-3}/above,
      120/{1}/above,
      180/{-2}/left,
      240/{3}/below,
      300/{-1}/below,
      360/{2}/right
     }
     \node[inner sep=1pt,circle,draw,fill,label={\p:\l}] at (\x:\R) {};
\end{tikzpicture}}
    \caption{The $A_2$ Coxeter complex.}
    \label{fig:hexagon}
\end{figure}

In Figure \ref{fig:convexSubcomplexes}, we depict the poset of proper non-empty convex subcomplexes of $\Sigma$.
Here, $\alpha_1$ is the root containing the vertices $1,-1,2,-3$, $\alpha_2$ is the root containing $1,2,-2,-3$, and $\alpha_3$ the root containing $1,-2,3,-3$,

\begin{figure}
    \centering
    \begin{tikzpicture}
	\draw[draw=black, fill=black, thin, solid] (-10.00,2.00) circle (0.1);
	\draw[draw=black, fill=black, thin, solid] (-8.50,2.00) circle (0.1);
	\draw[draw=black, fill=black, thin, solid] (-7.00,1.50) circle (0.1);
	\draw[draw=black, fill=black, thin, solid] (-5.50,2.00) circle (0.1);
	\draw[draw=black, fill=black, thin, solid] (-4.00,1.50) circle (0.1);
	\draw[draw=black, fill=black, thin, solid] (-2.50,2.00) circle (0.1);
	\draw[draw=black, fill=black, thin, solid] (-1.00,1.50) circle (0.1);
	\draw[draw=black, fill=black, thin, solid] (0.50,2.00) circle (0.1);
	\draw[draw=black, fill=black, thin, solid] (2.00,2.00) circle (0.1);
	\draw[draw=black, fill=black, thin, solid] (-9.00,4.50) circle (0.1);
	\draw[draw=black, fill=black, thin, solid] (-7.00,4.50) circle (0.1);
	\draw[draw=black, fill=black, thin, solid] (-5.00,4.50) circle (0.1);
	\draw[draw=black, fill=black, thin, solid] (-3.00,4.50) circle (0.1);
	\draw[draw=black, fill=black, thin, solid] (-1.00,4.50) circle (0.1);
	\draw[draw=black, fill=black, thin, solid] (1.00,4.50) circle (0.1);
	\draw[draw=black, fill=black, thin, solid] (-10.00,0.00) circle (0.1);
	\draw[draw=black, fill=black, thin, solid] (-8.50,0.00) circle (0.1);
	\draw[draw=black, fill=black, thin, solid] (2.00,0.00) circle (0.1);
	\draw[draw=black, fill=black, thin, solid] (0.50,0.00) circle (0.1);
	\draw[draw=black, fill=black, thin, solid] (-5.50,0.00) circle (0.1);
	\draw[draw=black, fill=black, thin, solid] (-9.00,-2.50) circle (0.1);
	\draw[draw=black, fill=black, thin, solid] (-7.00,-2.50) circle (0.1);
	\draw[draw=black, fill=black, thin, solid] (-5.00,-2.50) circle (0.1);
	\draw[draw=black, fill=black, thin, solid] (-3.00,-2.50) circle (0.1);
	\draw[draw=black, fill=black, thin, solid] (-1.00,-2.50) circle (0.1);
	\draw[draw=black, fill=black, thin, solid] (1.00,-2.50) circle (0.1);
	\node[black, anchor=south west] at (-9.65,4.65) {$\alpha_1$};
	\node[black, anchor=south west] at (-7.75,4.65) {$-\alpha_1$};
	\node[black, anchor=south west] at (-5.56,4.65) {$\alpha_2$};
	\node[black, anchor=south west] at (-3.73,4.65) {$-\alpha_2$};
	\node[black, anchor=south west] at (0.27,4.65) {$-\alpha_3$};
	\node[black, anchor=south west] at (-1.59,4.65) {$\alpha_3$};
	\draw[draw=black, thin, solid] (-9.00,4.50) -- (-7.00,1.50);
	\draw[draw=black, thin, solid] (-7.00,1.50) -- (-7.00,4.50);
	\draw[draw=black, thin, solid] (-5.00,4.50) -- (-4.00,1.50);
	\draw[draw=black, thin, solid] (-4.00,1.50) -- (-3.00,4.50);
	\draw[draw=black, thin, solid] (-1.00,4.50) -- (-1.00,1.50);
	\draw[draw=black, thin, solid] (-1.00,1.50) -- (1.00,4.50);
	\draw[draw=black, thin, solid] (-10.00,2.00) -- (-9.00,4.50);
	\draw[draw=black, thin, solid] (-10.00,2.00) -- (-5.00,4.50);
	\draw[draw=black, thin, solid] (-8.50,2.00) -- (-9.00,4.50);
	\draw[draw=black, thin, solid] (-5.00,4.50) -- (0.50,2.00);
	\draw[draw=black, thin, solid] (0.50,2.00) -- (-1.00,4.50);
	\draw[draw=black, thin, solid] (2.00,2.00) -- (1.00,4.50);
	\draw[draw=black, thin, solid] (-3.00,4.50) -- (2.00,2.00);
	\draw[draw=black, thin, solid] (-7.00,4.50) -- (-5.50,2.00);
	\draw[draw=black, thin, solid] (-2.50,2.00) -- (-1.00,4.50);
	\draw[draw=black, thin, solid] (-3.00,4.50) -- (-5.50,2.00);
	\draw[draw=black, thin, solid] (-7.00,4.50) -- (-2.50,2.00);
	\draw[draw=black, thin, solid] (-8.50,2.00) -- (1.00,4.50);
	\node[black, anchor=south west] at (-9.26,-3.22) {1};
	\node[black, anchor=south west] at (-7.39,-3.22) {-1};
	\node[black, anchor=south west] at (-5.26,-3.22) {2};
	\node[black, anchor=south west] at (-3.39,-3.22) {-2};
	\node[black, anchor=south west] at (-1.26,-3.22) {3};
	\node[black, anchor=south west] at (0.64,-3.22) {-3};
	\draw[draw=black, thin, solid] (-7.00,1.50) -- (-9.00,-2.50);
	\draw[draw=black, thin, solid] (-7.00,1.50) -- (-7.00,-2.50);
	\draw[draw=black, thin, solid] (-5.00,-2.50) -- (-4.00,1.50);
	\draw[draw=black, thin, solid] (-4.00,1.50) -- (-3.00,-2.50);
	\draw[draw=black, thin, solid] (-1.00,-2.50) -- (-1.00,1.50);
	\draw[draw=black, thin, solid] (-1.00,1.50) -- (1.00,-2.50);
	\draw[draw=black, thin, solid] (-10.00,0.00) -- (-10.00,2.00);
	\draw[draw=black, thin, solid] (0.50,2.00) -- (-10.00,0.00);
	\draw[draw=black, thin, solid] (-10.00,2.00) -- (-8.50,0.00);
	\draw[draw=black, thin, solid] (-8.50,0.00) -- (-8.50,2.00);
	\draw[draw=black, thin, solid] (-5.50,0.00) -- (-8.50,2.00);
	\draw[draw=black, thin, solid] (-5.50,0.00) -- (2.00,2.00);
	\draw[draw=black, fill=black, thin, solid] (-2.50,0.00) circle (0.1);
	\draw[draw=black, thin, solid] (-5.50,2.00) -- (-2.50,0.00);
	\draw[draw=black, thin, solid] (-2.50,0.00) -- (-2.50,2.00);
	\draw[draw=black, thin, solid] (-5.50,2.00) -- (2.00,0.00);
	\draw[draw=black, thin, solid] (2.00,0.00) -- (2.00,2.00);
	\draw[draw=black, thin, solid] (0.50,0.00) -- (0.50,2.00);
	\draw[draw=black, thin, solid] (-2.50,2.00) -- (0.50,0.00);
	\draw[draw=black, thin, solid] (-10.00,0.00) -- (-9.00,-2.50);
	\draw[draw=black, thin, solid] (-10.00,0.00) -- (1.00,-2.50);
	\draw[draw=black, thin, solid] (-8.50,0.00) -- (-5.00,-2.50);
	\draw[draw=black, thin, solid] (-8.50,0.00) -- (1.00,-2.50);
	\draw[draw=black, thin, solid] (-5.50,0.00) -- (-7.00,-2.50);
	\draw[draw=black, thin, solid] (-5.50,0.00) -- (-5.00,-2.50);
	\draw[draw=black, thin, solid] (-2.50,0.00) -- (-3.00,-2.50);
	\draw[draw=black, thin, solid] (-2.50,0.00) -- (-1.00,-2.50);
	\draw[draw=black, thin, solid] (0.50,0.00) -- (-9.00,-2.50);
	\draw[draw=black, thin, solid] (-3.00,-2.50) -- (0.50,0.00);
	\draw[draw=black, thin, solid] (2.00,0.00) -- (-1.00,-2.50);
	\draw[draw=black, thin, solid] (-7.00,-2.50) -- (2.00,0.00);
\end{tikzpicture}
    \caption{Poset $Y(\Sigma)_{\subsetneq \Sigma}$ of proper and non-empty convex subcomplexes of $\Sigma$.}
    \label{fig:convexSubcomplexes}
\end{figure}
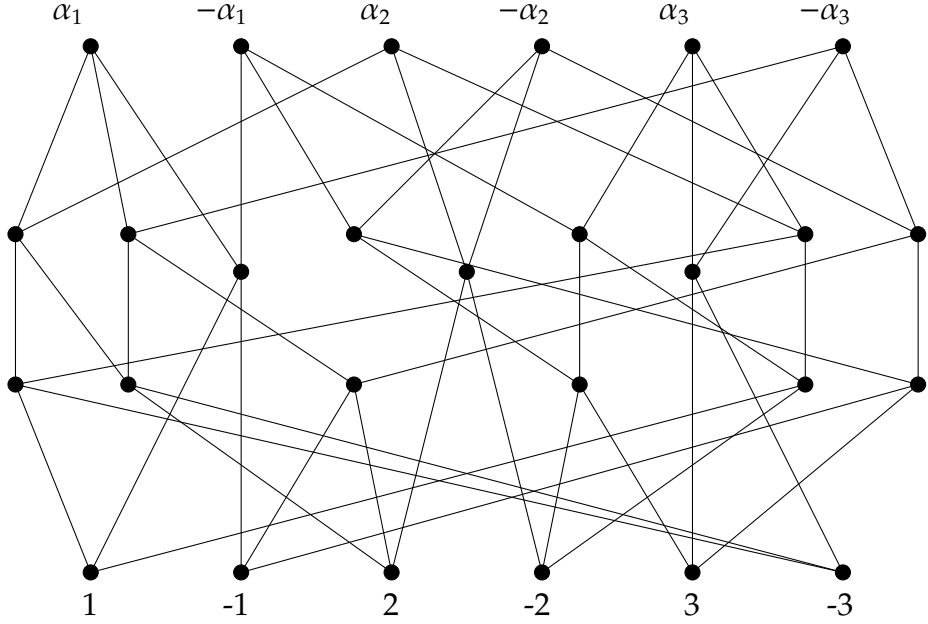

Consider the convex subcomplex $K = \alpha_1$.
As seen in Figure \ref{fig:lowerIntervalConvexKinA2}, $Y(\Sigma)_{\subsetneq K}$ has dimension two, but it collapses to the one-dimensional subposet obtained by keeping the points $\alpha_1\cap \alpha_2$, $\alpha_1\cap -\alpha_3$, $\alpha_1\cap-\alpha_1$ and $\alpha_1\cap\alpha_2\cap -\alpha_3$, $1$ and $-1$, which triangulates a one-dimensional sphere.
Therefore, $Y(\Sigma)_{\subsetneq K}$ is not spherical.

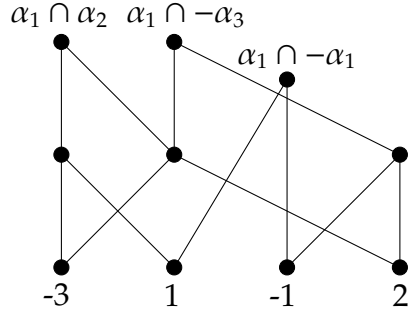
\begin{figure}
    \centering
    \begin{tikzpicture}
	\draw[draw=black, fill=black, thin, solid] (-10.00,1.50) circle (0.1);
	\draw[draw=black, fill=black, thin, solid] (-8.50,1.50) circle (0.1);
	\draw[draw=black, fill=black, thin, solid] (-7.00,1.00) circle (0.1);
	\draw[draw=black, fill=black, thin, solid] (-10.00,0.00) circle (0.1);
	\draw[draw=black, fill=black, thin, solid] (-8.50,0.00) circle (0.1);
	\draw[draw=black, fill=black, thin, solid] (-5.50,0.00) circle (0.1);
	\draw[draw=black, fill=black, thin, solid] (-10.00,-1.50) circle (0.1);
	\draw[draw=black, fill=black, thin, solid] (-7.00,-1.50) circle (0.1);
	\draw[draw=black, fill=black, thin, solid] (-5.50,-1.50) circle (0.1);
	\draw[draw=black, fill=black, thin, solid] (-8.50,-1.50) circle (0.1);
	\node[black, anchor=south west] at (-8.78,-2.16) {1};
	\node[black, anchor=south west] at (-7.38,-2.16) {-1};
	\node[black, anchor=south west] at (-5.74,-2.16) {2};
	\node[black, anchor=south west] at (-10.39,-2.16) {-3};
	\draw[draw=black, thin, solid] (-7.00,1.00) -- (-8.50,-1.50);
	\draw[draw=black, thin, solid] (-7.00,1.00) -- (-7.00,-1.50);
	\draw[draw=black, thin, solid] (-10.00,0.00) -- (-10.00,1.50);
	\draw[draw=black, thin, solid] (-10.00,1.50) -- (-8.50,0.00);
	\draw[draw=black, thin, solid] (-8.50,0.00) -- (-8.50,1.50);
	\draw[draw=black, thin, solid] (-5.50,0.00) -- (-8.50,1.50);
	\draw[draw=black, thin, solid] (-10.00,0.00) -- (-8.50,-1.50);
	\draw[draw=black, thin, solid] (-10.00,0.00) -- (-10.00,-1.50);
	\draw[draw=black, thin, solid] (-8.50,0.00) -- (-5.50,-1.50);
	\draw[draw=black, thin, solid] (-8.50,0.00) -- (-10.00,-1.50);
	\draw[draw=black, thin, solid] (-5.50,0.00) -- (-7.00,-1.50);
	\draw[draw=black, thin, solid] (-5.50,0.00) -- (-5.50,-1.50);
	\node[black, anchor=south west] at (-10.82,1.55) {$\alpha_1\cap \alpha_2$};
	\node[black, anchor=south west] at (-9.28,1.55) {$\alpha_1\cap -\alpha_3$};
	\node[black, anchor=south west] at (-7.80,1) {$\alpha_1\cap -\alpha_1$};
\end{tikzpicture}
    \caption{Lower interval $Y(\Sigma)_{\subsetneq K}$.}
    \label{fig:lowerIntervalConvexKinA2}
\end{figure}

Now suppose that $K$ is the Levi sphere $\alpha_1\cap-\alpha_1$.
As seen in Figure \ref{fig:convexSubcomplexes}, the upper interval $Y(\Sigma)_{\supsetneq K}$ is the poset $\{\alpha_1,-\alpha_1, \Sigma\}$ (of dimension one), and the lower interval $Y(\Sigma)_{\supsetneq K} = \{1, -1\}$ is discrete (of dimension zero).
Since the poset $Y(\Sigma)$ has dimension four, the ``correct" dimension for $\Lk_{\K(Y(\Sigma))}(K) = \K(Y(\Sigma)_{\subsetneq K}) * \K(Y(\Sigma)_{\supsetneq K})$ should be three, but in this case we get a complex of dimension two.
As $K$ has height one in $Y(\Sigma)$, the expected dimension of $Y(\Sigma)_{\supsetneq K}$ is $\dim Y(\Sigma) - 1 - 1 = 2$.
\end{example}

In Theorem \ref{thm:PDandCB}, we will see that $Y(\Delta)$ has the homotopy type of a poset of dimension $2\dim \Delta + 1$, so it cannot be Cohen-Macaulay.

\section{Levi spheres and decompositions}
\label{sec:leviSpheres}
~
Based on Theorem \ref{thm:decompAndLeviVectorSpaces}, we define the poset of decompositions of a building $\Delta$ as the poset of Levi spheres.
Recall, that a Levi sphere is an intersection of walls in some apartment (and hence in any apartment containing it).
In particular, apartments are Levi spheres.

\begin{definition}
\label{def:decompBuilding}
Let $\Delta$ be a spherical building.
We denote by $\D(\Delta)$ the poset of all non-empty Levi spheres, with order induced by reverse inclusion.
\end{definition}

Note that the empty subcomplex $\emptyset$ can be regarded as a Levi sphere, but we will not include it in $\D(\Delta)$.
We write $\overline{\D}(\Delta) = \D(\Delta)\cup \{\emptyset\}$, so $\emptyset$ is a unique maximal element for this poset.

We will analyze upper intervals in $\D(\Delta)$.
Indeed, these are related to hyperplane arrangements and parabolic subgroups of the underlying Weyl group, as we explain below.

If $\Sigma$ is the Coxeter complex associated with the Coxeter group $W$ with set of simple reflections given by $R$, we see that the Levi spheres contained in $\Sigma$ correspond to intersections of hyperplanes defined by the hyperplane arrangement associated with $(W,R)$, which we may denote by $\H(W,R)$.
Write $L(W,R)$ for the geometric lattice associated with the arrangement $\H(W,R)$, whose elements are intersections of hyperplanes in $\H(W,R)$ ordered by reverse inclusion.
By convention, the ``empty" intersection of hyperplanes is the full subspace spanned by the elements of $\H(W,R)$, which is the unique minimal element of the poset $L(W,R)$.
Thus, non-empty Levi spheres contained in $\Sigma$ are in one-to-one correspondence with non-maximal elements of $L(W,R)$.

On the other hand, a parabolic subgroup of $W$ is a $W$-conjugate of a subgroup of the form $\gen{I}$, where $I\subseteq R$.
We denote by $\P(W,R)$ the poset of parabolic subgroups of $W$ ordered by reverse inclusion.
Note that $W$ is the unique minimal element of $\P(W,R)$, and the trivial subgroup is the unique maximal element.
By \cite[Theorem 3.1]{BIhyperplanes}, $L(W,R) \cong \P(W,R)$.

Thus, we conclude:

\begin{proposition}
\label{prop:upperIntervalsDecompDelta}
Let $\Delta$ be a spherical building and let $\Sigma\in \A(\Delta)$ be an apartment.
Suppose that $\Sigma$ is the Coxeter complex associated with the Coxeter group $(W,R)$, where $R$ is a set of simple reflections.
Then
\[ \overline{\D}(\Delta)_{\geq \Sigma} \cong L(W,R) \cong \P(W,R) .\]

In particular, if $S$ is a Levi sphere of $\Delta$ contained in $\Sigma$, then $\overline{\D}(\Delta)_{\geq S}$ is an upper-interval of the geometric lattice $L(W,R)$.
Thus, $\D(\Delta)_{>S}$ is spherical of dimension $\dim S - 1$.
\end{proposition}

For lower intervals, we have the following description.

\begin{lemma}
\label{lem:lowerIntervalsDecompDelta}
Let $\Delta$ be a spherical building, and let $S\in \D(\Delta)$ be a Levi sphere.
Then $\D(\Delta)_{\leq S} \cong \D(\Lk_{\Delta}(\sigma))$, where $\sigma\in S$ is a maximal simplex.
The isomorphism is given by mapping $S'\in \D(\Delta)_{\leq S}$ to the Levi sphere obtained as the convex hull of $\proj_{\sigma}(\tau),\proj_{\sigma}(\tau')$ in $\Lk_{\Delta}(\sigma)$, where $S'$ is the convex hull of the opposite pair $(\tau,\tau')$.
\end{lemma}

\begin{proof}
See \cite[p. 199]{Serre}.
\end{proof}

In Theorem \ref{thm:DisCM} we prove that $\D(\Delta)$ is Cohen-Macaulay.

Next, we define a ``partial decomposition poset" for buildings using our interpretation of decompositions in terms of Levi spheres, inspired by Theorems \ref{thm:decompAndLeviVectorSpaces} and \ref{thm:PDandPDBuildingVectorSpace}.

\begin{definition}
\label{def:PDbuildings}
Let $\Delta$ be a spherical building.
We define $\PD(\Delta)$ as the poset on the disjoint union of 
$\X(\Delta)$ and ${\D(\Delta)}$ with order relation $\preceq$
defined as follows. 
Among elements of $\X(\Delta)$ and among elements of $\D(\Delta)$ the
order $\preceq$ is the one inherited from the order on the 
respective posets.
If $\sigma \in \Delta$ and $S\in {\D(\Delta)}$, then we set $\sigma \prec S$ if and only if there exists an apartment $\Sigma$ of $\Delta$ such that $\sigma\in \Sigma$ and $S$ is a Levi sphere of $\Sigma$.
\end{definition}

It is not hard to see that $\PD(\Delta)$ with the relation $\preceq$
is indeed a poset.

\begin{remark}
[$\CB$ and $\PD$ for arbitrary posets]
In \cite{BPW24}, the authors propose a definition of the common bases complex and the partial decomposition poset for rather arbitrary posets.

According to \cite[Definition 2.1]{BPW24}, to define the poset $\PD$ and the complex $\CB$ associated with a poset $\P$, one must first select a family of \textit{frames} of  $\P$.
For example, following \cite{BPW24}, a maximal simplex for the common bases complex associated with $\P$ and the fixed family of frames, is a set whose elements are all possible joins of proper and non-empty subsets of a frame.

When one starts with the poset of subspaces of a finite-dimensional vector space $V$, then the definitions in \cite{BPW24}
yield the poset of partial direct sum decompositions from Definition \ref{def:DandPDVectorSpaces} 
and the common basis complex from Definition \ref{def:CBVectorSpaces}.

In the more general context of buildings, apartments play the 
role of the maximal simplices in the common bases complex.
However, we do not see how one can extend the definitions
from \cite{BPW24} and define the concept of a frame
within an apartment, as this is now an abstract Coxeter complex.
Therefore, the notion of ``frame" as a certain ``generating set of vertices" is less clear in the context of buildings.
Similarly, we do not see how to use the definition from \cite{BPW24}
to define a partial decomposition poset for buildings.
\end{remark}

The main result, Theorem 2.9, of \cite{BPW24} says that under mild hypotheses their common bases complex and their poset of partial decompositions 
are homotopy equivalent.
In the context of buildings, we also prove that the common bases complex and the partial decomposition poset are homotopy equivalent (cf. Theorem \ref{thm:PDandCBVectorSpaces}).

Recall from Lemma \ref{lm:CBequivYDelta} that the inclusion $Y(\Delta)\hookrightarrow \X(\CB(\Delta))$ is an equivariant homotopy equivalence.
Also, $Y(\Delta)$ is the poset of non-empty convex subcomplexes of $\Delta$ that lie in some apartment, ordered by inclusion (see Definition \ref{def:convexSubcomplexes}).
In what follows, $\PD(\Delta)_{\prec \emptyset} = \PD(\Delta)$ and
$Y(\Delta)_{\supseteq \emptyset} = Y(\Delta)$.

\begin{theorem}
\label{thm:PDandCB}
Let $\Delta$ be a spherical building and let $H$ be a group acting on $\Delta$ by simplicial automorphisms.
Suppose that $S$ is a Levi sphere invariant under $H$ (possibly $S = \emptyset$).

Then $\Gamma:(\PD(\Delta)_{\prec S})' \to Y(\Delta)_{\supseteq S}$ given by $\Gamma(c) = \Conv_{\Delta}(c,S)$ is a well-defined order-preserving $H$-equivariant map.

Moreover, the following hold:

\begin{enumerate}
    \item If $S = \emptyset$ then $\Gamma$ induces a homotopy equivalence $(\PD(\Delta)')^H \to Y(\Delta)^H$.
    In particular, $\Gamma$ gives rise to a homotopy equivalence
    \[ \PD(\Delta)\simeq Y(\Delta)\simeq \CB(\Delta),\]
    which is $H$-equivariant if $H$ is a compact Lie group, and to an
    $H$-equivariant isomorphism 
    \[ \widetilde{H}_*(\PD(\Delta))\cong \widetilde{H}_*(\CB(\Delta)).\]
    \item $\PD(\Delta)^H_{\prec S} \simeq S^H * Y(\Delta)^H_{\supsetneq S}$.
    \item If $\sigma\in S$ is a simplex of maximal dimension, then $\PD(\Delta)_{\prec S} \simeq |S| * Y(\Lk_{\Delta}(\sigma))$.
\end{enumerate}
\end{theorem}

\begin{proof}
First, we show that $\Gamma$ is well-defined.
Note that if $c\in \PD(\Delta)'$, then there is an apartment $\Sigma$ containing all the vertices involved in $c$ by definition of $\PD(\Delta)$.
In particular, the convex hull of $c$ is a convex subcomplex of $\Sigma$, so it lies in $Y(\Delta)$.
If in addition the maximal element of $c$, say $x$, satisfies $x\prec S$, we can take $\Sigma$ containing $S$ by definition of the ordering in $\PD(\Delta)$, and hence $\Gamma(c) = \Conv_{\Delta}(c,S)\in Y(\Delta)_{\supseteq S}$.
Thus $\Gamma$ is an order-preserving map between posets.
Also, as pointed out at the end of Subsection \ref{subsub:chambercomplexes}, $\Gamma$ is $H$-equivariant.

Let $Y := Y(\Delta)_{\supseteq S}$, and let $X := (\PD(\Delta)_{\prec S})'$.
From now on, we work with the induced map $\Gamma: X\to Y$, and we also denote $\Gamma_H:X^H\to Y^H$ the map induced on the fixed-point subposets.

We use Quillen's fiber theorem to prove that $\Gamma_H$ is a homotopy equivalence.
Let $K\in Y^H$, so $K$ is a convex subcomplex of some apartment $\Sigma$.
We prove that $\Gamma_H^{-1}( Y^H_{\subseteq K} )$ is contractible.

Now, note that $\Gamma(c) \subseteq K$ if and only if every simplex involved in $c$ is a simplex of $K$.
Therefore, the $\Delta$-part of the chain $c \subseteq \X(\Delta) \cup {\D(\Delta)}$ is a set of simplices of $K$, and the $\D(\Delta)$-part of $c$ is a set of Levi spheres also completely contained in $K$ but that properly contain $S$.
Write
\[ \D_K := \{ T \in {\D(\Delta)_{\prec S}} \tq T\subseteq K\}. \]
Then $\D_K$ is the set of Levi spheres of $\Sigma$ that are contained in $K$ and properly contain $S$.
This shows that
\[\Gamma^{-1}(Y_{\subseteq K}) = \big(\X(K) \ojoin \D_K\big)'.\]
Next, a chain $c\in (\X(K) \ojoin \D_K)'$ is fixed by $H$ if and only if each element of $c$ is $H$-invariant. That is,
\begin{equation}
\label{eq:preimageGammaFixPoints}
\Gamma^{-1}_H(Y_{\subseteq K}^H) = \Gamma^{-1}(Y_{\subseteq K})^H = \left(\X(K)^H \ojoin \D_K^H\right)',
\end{equation}
where $\X(K)^H$ is the poset of non-empty simplices of $K$ fixed by $H$ (but maybe not pointwise fixed), and $\D_K^H$ consists of the $H$-invariant Levi spheres of $\Sigma$ that are contained in $K$ and properly contain $S$.
Thus, it is enough to prove that either $K^H$ or $\D_K^H$ is contractible.
Recall that $|K^H| = | (K^H)'| = |\X(K)^H|$.

Suppose first that we take $K = S$, so in particular $S\neq\emptyset$.
Then $\D_K = \emptyset$, and
\[ \Gamma^{-1}_H(Y^H_{\subseteq K}) = \Gamma^{-1}(S)^H = \X(S)^H.\]
This is not a contractible poset in general, but its dimension is at most $\dim S$.

Now suppose that $K\supsetneq S$.
We prove that in any case, the preimage $\Gamma^{-1}_H(Y^H_{\subseteq K})$ is contractible.
We split the analysis of the homotopy type of the preimage into two cases:

\medskip
\noindent
\textbf{Case 1.} If $K$ contains a pair of maximal opposite simplices $\sigma,\sigma'$, then $\D_K^H$, and also $\Gamma^{-1}_H(Y^H_{\subseteq K})$, are contractible.

\medskip
\noindent {\it Proof.}
Since $K$ lies in an apartment $\Sigma$, the simplex $\sigma$ has exactly one opposite in $K$.
Thus $K = \Conv_{\Sigma}(\sigma,\sigma')$ by Lemma \ref{lm:convexIsleviSphere}.
In particular, $K$ is a unique minimal element of $\D_K^H$, so it is a contractible poset.
This proves that the preimage in Equation \eqref{eq:preimageGammaFixPoints} is a contractible poset.
This case follows.    
\noindent 
$\Box$

\bigskip
\noindent
\textbf{Case 2.} If $K$ does not contain a pair of maximal opposite simplices, then $\X(K)^H$, and so $\Gamma^{-1}_H(Y^H_{\subseteq K})$, are contractible.

\medskip
\noindent {\it Proof.}
By \cite[Théorème 2.1]{Serre}, $K$ is a contractible convex subcomplex.
We claim that $K^H$ is also contractible.

Let $\Sigma$ be an apartment containing $K$.
Regard the geometric realization of $\Sigma$ as the unit sphere via the usual CAT(1) metric, denoted by $\dcat$ (see Subsection \ref{subsub:cat1metric}).
Recall that $K$, being a convex subcomplex of $\Sigma$, is obtained by intersecting roots.
In particular, we can find a sphere $\mathbb{S}^m$ such that $m = \dim K$ and $|K|\subseteq \mathbb{S}^m \subseteq |\Sigma|$.
Indeed, $\mathbb{S}^m = |\tilde{K}|$ where $\tilde{K}$ is the intersection of all walls in $\Sigma$ containing $K$ (so $\widetilde{K}$ is a Levi sphere).
Let $C = |K|$.
Then $C$ is a proper closed convex subset of $|\tilde{K}| = \mathbb{S}^m$ with non-empty interior (as they have the same dimension).
By \cite[Corollary 2.7]{ramoscuevas}, $C$ has a unique point $x_0\in C$ such that
\begin{equation}
\label{eq:capx0}
d(x_0, \partial C) = \sup_{x\in C} d(x,\partial C),
\end{equation}
and
\begin{equation}
\label{eq:distancex0}
    \sup_{x\in C} d(x,x_0) \leq \frac{\pi}{2}.
\end{equation}
Since $x_0$ is uniquely defined in terms of the metric property Equation \eqref{eq:capx0}, and $H$ acts by isometries on $C$, we see that $x_0$ is fixed by $H$, so $x_0\in C^H$.
In addition, $C$ cannot contain a point opposite to $x_0$ by Equation \eqref{eq:distancex0} (recall that opposite points are at distance $\pi$).
Thus, for every $y\in C$, there is a unique geodesic joining $y$ with $x_0$, and such a geodesic must be contained in $C$ since $K$ is a convex subcomplex.
In particular, if $y\in C^H$, then by uniqueness, such a geodesic must be completely contained in $C^H$.
Hence, $|K^H| = C^H$ is contractible.
By Equation \eqref{eq:preimageGammaFixPoints}, the preimage $\Gamma^{-1}_H( Y^H_{\subseteq K})$ is contractible.
This finishes the proof of Case 2.
\noindent 
$\Box$

\bigskip

Now we prove the assertions (1) - (3).
The first part of Item (1) follows from Quillen's fiber Theorem \ref{thm:quillen}(2).
The ($H$-equivariant) homotopy equivalences $\PD(\Delta)\simeq Y(\Delta)\simeq \CB(\Delta)$ then follow from Theorem \ref{thm:whitehead} and Lemma \ref{lm:CBequivYDelta}.
The homology isomorphism is induced by $\Gamma$, and it is $H$-equivariant since $\Gamma$ is.

Item (3) follows from item (2) by taking $H = 1$ and Proposition \ref{prop:intervalSphereYPoset}.

Thus, it remains to prove item (2).
We can assume that $S\neq\emptyset$ since the case $S = \emptyset$ is item (1).
We show that
\begin{equation}
\label{eq:PDandJoin}
\PD(\Delta)_{\prec S}^H \simeq S^H * Y(\Delta)^H_{\supsetneq S}.
\end{equation}
First, note that $S$ is the unique minimal element of $Y^H$.
Next, observe that for $K\in Y^H$, $K\neq S$, the map
\[ \Gamma^{-1}_H(Y^H_{\subsetneq K}) \hookrightarrow \Gamma^{-1}_H(Y^H_{\subseteq K})\]
is null-homotopic by Cases 1 and 2 since the codomain of this inclusion is a contractible space.
For the preimage of the minimal element $S\in Y^H$, we have
\[ \Gamma^{-1}_H(S) = S^H.\]
If this preimage is non-empty, the homotopy equivalence in Equation \eqref{eq:PDandJoin} follows from Theorem \ref{thm:wedgeDecomposition}.

If $S^H$ is empty, then we consider instead $\Gamma_H : X^H \to Y(\Delta)^H_{\supsetneq S}$, and every preimage of a lower interval $Y^H_{\subseteq K} \cap Y^H_{\supsetneq S}$ is still contractible (as the preimage of $S$ does not contribute).
In that case, $\Gamma_H$ is a homotopy equivalence by Theorem \ref{thm:quillen}, so Equation \eqref{eq:PDandJoin} holds.
This concludes the proof of item (2).
\end{proof}

\section{An ordered partial decomposition poset for buildings}
\label{sec:opd}

In this section, we define the poset $\OPD(\Delta)$ of ``ordered partial decompositions" of a spherical building $\Delta$.
This poset does not resemble the poset of ordered partial decompositions in the case $\Delta$ is the building of a finite-dimensional vector space, but it does so up to equivariant homotopy, as we will see later.
This homotopy equivalence is the motivation for the definition we give here.

\begin{definition}
\label{def:ODandOPDbuildings}
Let $\Delta$ be a spherical building.
We define the posets:
\begin{itemize}
\item $\OD(\Delta) := \Opp(\Delta) = \{ (\sigma,\sigma')\tq \sigma,\sigma'$ are opposite simplices$\}$, with order relation 
$(\sigma,\sigma') \leq (\tau,\tau')$ if $\tau\subseteq \sigma$ and $\tau'\subseteq \sigma'$.
\item $\OPD(\Delta) = \X(\Delta) \cup \Opp(\Delta)$ with order relation $\preceq$ defined as follows. Among elements of $\X(\Delta)$ and 
among elements of $\Opp(\Delta)$ the order $\preceq$ is inherited from the order on the respective posets. If $\sigma \in \X(\Delta)$ and
$(\tau,\tau') \in \Opp(\Delta)$ then we set 
$\sigma \prec (\tau,\tau')$ if there is an apartment containing $\sigma,\tau$, and $\tau'$.
\end{itemize}
We call $\OD(\Delta)$ the ordered decomposition poset associated with $\Delta$, and $\OPD(\Delta)$ the ordered partial decomposition poset.
\end{definition}

Note that $\OD(\Delta)$ has dimension $\dim \Delta$.
The poset $\OD(\Delta)^\op$ was first introduced by Lehrer and Rylands in \cite{LR} under the name split building when $\Delta = \Delta(\GG,k)$ is the building of a connected reductive group $\GG$ defined over a field $k$  (cf. Definition \ref{def:DecompPosetsAlgebraicGroups}).
In \cite{LR} it is proved that $\OD(\Delta(\GG,k))$ is spherical when $\GG(k)$ is a classical group, and conjectured that this holds 
for an arbitrary connected reductive group $\GG$.
This conjecture was settled by von Heydebreck in the more general context of buildings \cite{vH}.

\begin{theorem}
[von Heydebreck]
\label{thm:vonHeydebreck}
Let $\Delta$ be a spherical building of dimension $n$.
Then $\OD(\Delta)$ is Cohen-Macaulay of dimension $n$.
\end{theorem}

The following lemma describes the intervals of the poset $\OD(\Delta)$, which was first done in \cite{vH}.
Recall that for two simplices $\sigma,\tau\in \Delta$, the projection $\proj_{\sigma}(\tau)$ is the unique maximal element in the convex hull of $\{\sigma,\tau\}$ that contains $\sigma$.
In particular, we have a simplicial map $\proj_{\sigma}:\Delta\to \Lk_{\Delta}(\sigma)$.
See \cite[2.29 \& 3.19]{Tits1}.

\begin{lemma}
\label{lm:intervalsBuildingCase}
Let $\Delta$ be a spherical building, and let $(\sigma,\sigma')\in \OD(\Delta)$.
Then the following hold:
\begin{enumerate}
\item There is an isomorphism of posets $\OD(\Delta)_{>(\sigma,\sigma')} \to \{ \tau \subsetneq \sigma \tq \emptyset\neq \tau \}^{\op}$ given by $(\tau,\tau')\mapsto \tau$.
In particular, $\OD(\Delta)_{>(\sigma,\sigma')}$ is a sphere of dimension $\dim \sigma -1$.
\item There is an isomorphism $\OD(\Delta)_{<(\sigma,\sigma')} \to \OD( \Lk_\Delta(\sigma))$ defined by $(\tau,\tau')\mapsto (\tau,\proj_{\sigma}(\tau'))$.
In particular, $\OD(\Delta)_{<(\sigma,\sigma')}$ has the homotopy type of a wedge of spheres of dimension $\codim \sigma - 1$.
\end{enumerate}
Moreover, these maps are $\Stab_{\Aut(\Delta)}(\sigma)$-equivariant.
\end{lemma}

\begin{proof}
To prove (1), fix an apartment $\Sigma$ containing $\sigma,\sigma'$.
Then, elements $(\tau,\tau')\in \OD(\Delta)_{>(\sigma,\sigma')}$ are uniquely determined by $\emptyset \neq \tau \subsetneq \sigma$ and $\tau' = \op_{\Sigma}(\tau)$.
This shows the isomorphism described in item (1).

Item (2) is explained in Proposition 2.1 from \cite{vH}.
The equivariance of the maps follows immediately from their definitions. 
\end{proof}

We can define an order-preserving map $F: \OD(\Delta)\to \D(\Delta)$ by
\begin{equation}
\label{eq:ForgetfulMap}
 F(\sigma,\sigma') = \Conv_{\Delta}(\sigma,\sigma').
\end{equation}

Using $F$, we show that $\D(\Delta)$ is also Cohen-Macaulay.

\begin{theorem}
\label{thm:DisCM}
Let $\Delta$ be a spherical building.
Then we have a homotopy equivalence
\begin{equation}
\label{eq:wedgeODandD}
\OD(\Delta)\simeq \D(\Delta) \vee \bigvee_{S\in \D(\Delta)} |S| * \D(\Lk_{\Delta}(\sigma_S)).    
\end{equation}
where $\sigma_S\in S$ is some maximal simplex.
Moreover, $\D(\Delta)$ is a Cohen-Macaulay poset.
\end{theorem}

\begin{proof}
By Proposition \ref{prop:upperIntervalsDecompDelta}, Lemma \ref{lem:lowerIntervalsDecompDelta} and induction on $\dim \Delta$, every interval of $\D(\Delta)$ is spherical of the right dimension.
It remains to show that $\D(\Delta)$ is spherical and that the wedge decomposition in Equation \eqref{eq:wedgeODandD} holds.
To that end, we use the map $F:\OD(\Delta)\to \D(\Delta)$ defined in Equation \eqref{eq:ForgetfulMap} and Theorem \ref{thm:wedgeDecomposition}.
By Lemma \ref{lem:lowerIntervalsDecompDelta} and induction, for every $S\in \D(\Delta)$, the lower interval $ \D(\Delta)_{<S}$ is spherical of dimension $\dim \Lk_{\Delta}(\sigma)$, where $\sigma\in S$ is a maximal simplex.

On the other hand, $F^{-1}(\D(\Delta)_{\geq S})$ consists of the pairs of opposite simplices contained in $S$.
This identifies with the subcomplex $S$: every simplex $\sigma \in S$ has a unique opposite $\op_S(\sigma)$ in $S$, so we have a simplicial isomorphism
\[ \sigma \in S \mapsto (\sigma, \op_S(\sigma)) \in F^{-1}(\D(\Delta)_{\geq S}).\]
That is, $\X(S)^{\op}\cong F^{-1}(\D(\Delta)_{\geq S})$ as posets.
Hence, $F^{-1}(\D(\Delta)_{>S}) \hookrightarrow F^{-1}(\D(\Delta)_{\geq S})$ is a null-homotopic map.
By Theorem \ref{thm:wedgeDecomposition}, Equation \eqref{eq:wedgeODandD} holds.
Since $\OD(\Delta)$ is spherical by Theorem \ref{thm:vonHeydebreck} and $\dim \OD(\Delta) = \dim\Delta = \dim \D(\Delta)$, we conclude that $\D(\Delta)$ is spherical (and therefore Cohen-Macaulay).
\end{proof}

Note that Theorem \ref{thm:DisCM} implies Corollary \ref{coro:DDeltaCM}

Next, we prove that $\OPD(\Delta)$ has the homotopy type of the simplicial join $\Delta*\Delta$, and that it carries the tensor-square Steinberg representation in homology.
For a poset $X$, recall that $X^{(i)}$ denotes the subposet of elements of height at most $i$.
If $X = \X(K)$ where $K$ is a Cohen-Macaulay simplicial complex, then $X$ is Cohen-Macaulay as a poset.
Also, every rank selection of a Cohen-Macaulay poset is spherical (see Theorem 6.2 of \cite{Baclawski}).
In particular, rank selections of the face poset of a (spherical) building are Cohen-Macaulay.

\begin{theorem}
\label{thm:OPDandDeltaDelta}
Let $\Delta$ be a spherical building of dimension $m$ and let $H$ be a group acting on $\Delta$ by simplicial automorphisms.
Write $D_1  = \X(\Delta)$ and $D_2 = \X(\Delta)^{\op}$.
Note that $\OPD(\Delta) = D_1 \cup \OD(\Delta)$.

Let $\phi:\OPD(\Delta)\to D_1 \ojoin D_2$ be the following map:
\begin{align*}
    \phi(\sigma,\sigma') & = \sigma \in D_2 \quad \text{ if } (\sigma,\sigma')\in \OD(\Delta) \subset \OPD(\Delta),\\
    \phi(\sigma) & = \sigma \in D_1 \quad \text{ if }  \sigma\in D_1 \subset \OPD(\Delta).
\end{align*}
Then $\phi$ is an $H$-equivariant order-preserving map.
Moreover, the following hold:
\begin{enumerate}
\item The restriction
\[ \phi_H: (\OPD(\Delta)^{(i+m+1)})^H \to  (D_1 \ojoin (D_2^{(i)}))^H\]
is a homotopy equivalence for all $0 \leq i \leq m$.
Thus, $\OPD(\Delta)^{(i+m+1)} \simeq D_1 \ojoin D_2^{(i)}$ is spherical of dimension $m + i + 1$.
\item In particular, $\phi$ gives rise to a homotopy equivalence
\[\OPD(\Delta) \simeq \Delta * \Delta,\]
which is $H$-equivariant if $H$ is a compact Lie group, and to an $H$-equivariant isomorphism
\[ \widetilde{H}_{*}(\OPD(\Delta)) \cong_H \widetilde{H}_*(\Delta*\Delta).\]
\end{enumerate}
\end{theorem}

\begin{proof}
Write $X = \OPD(\Delta)^{(i+m+1)}$ and $Y = D_1\ojoin (D_2^{(i)})$.
As $\phi$ is clearly order-preserving when restricted to $D_1$ and to $\OD(\Delta_1)$, it defines an order-preserving map $\phi:\OPD(\Delta) \to D_1\ojoin D_2$.
It is also clear that $\phi$ is $H$-equivariant.
Moreover, $\phi$ restricts to an $H$-equivariant order-preserving map $\phi:X\to Y$.

Applying item (1) when $H = 1$ and $i = m$
yields the homotopy equivalence. Observing that the map on homology induced by a homotopy equivalence is always an isomorphism, we derive item (2) up to $H$-equivariance. The $H$-module isomorphism of homology groups then follows from (1) and
Whitehead's Theorem \ref{thm:whitehead}.

To prove item (1), we use Quillen's fiber Theorem \ref{thm:quillen} on the induced poset map $\phi_H: X^H\to Y^H$ and show that for all $\tau \in Y^H$, the preimage $\phi^{-1}_H(Y^H_{\leq \tau})$ is contractible.
Note that $X^H = D_1^H \cup ( \OD(\Delta)^{(i)})^H$, and $Y^H = D_1^H \ojoin (D_2^{(i)})^H$.

If $\tau \in D_1^H\subseteq Y^H$, then 
\[ \phi^{-1}_H \big(Y^H_{\leq \tau} \big) = {D_1}_{\subseteq \tau}\]
is the poset of non-empty faces of $\tau$, which is contractible as $\tau$ is its unique maximal element.

Now take $\tau\in D_2^H$ of height at most $i$ in $D_2$.
We show that $\phi^{-1}_H(Y^H_{\leq \tau})$ is contractible.
For $\Sigma\in \A(\Delta,\tau)$, set $V_{\Sigma} := \X(\Sigma)^H \ojoin (\OD(\Sigma)_{\leq (\tau, op_{\Sigma}
(\tau))})^H$.

\smallskip

\noindent {\bf Claim:} We have the following equality:
\begin{equation}
\label{eq:preimageAsUnion}
\phi^{-1}_H \big(Y^H_{\leq \tau} \big)  = \bigcup_{\Sigma \in \A(\Delta,\tau)}  V_{\Sigma} = \bigcup_{\Sigma \in \A(\Delta,\tau)} \X(\Sigma)^H \ojoin (\OD(\Sigma)_{\leq (\tau, op_{\Sigma}(\tau))})^H.
\end{equation}

\smallskip 

\noindent {\textit{Proof:}}
If $\sigma\in \phi^{-1}_H(Y^H_{\leq \tau}) \cap D_1^H$, then there is an apartment $\Sigma\in \A(\Delta)$ containing both $\sigma$ and $\tau$, so $\sigma\in V_{\Sigma}$.
If $(\sigma,\sigma')\in \phi^{-1}_H(Y^H_{\leq \tau}) \cap \OD(\Delta)^H$ then $\sigma \supseteq \tau$ and there is an apartment $\Sigma$ containing $\sigma$ and $\sigma'$, and hence $\tau$.
Clearly $\op_\Sigma(\tau) \subseteq \sigma'$, i.e. $(\sigma,\sigma')\leq (\tau, \op_{\Sigma}(\tau))$, and $(\sigma,\sigma')\in (\OD(\Sigma)_{\leq (\tau, op_{\Sigma}
(\tau))})^H \subseteq V_{\Sigma}$ with $\Sigma \in\A(\Delta,\tau)$.
Note that we are not claiming that $(\tau,\op_{\Sigma}(\tau))\in V_{\Sigma}$ as $H$ may not fix $\op_{\Sigma}(\tau)$.
On the other hand, it is not hard to verify that $V_{\Sigma}\subseteq \phi^{-1}_H(Y^H_{\leq \tau})$ for all $\Sigma\in \A(\Delta,\tau)$.
This concludes with the proof of the claim. $\Box$

\smallskip

Now we show that any intersection of finitely many $V_{\Sigma}$s is contractible.
This implies that $\phi^{-1}_H(Y^H_{\leq \tau})$ is contractible by, for instance, the nerve theorem.
Let $\Sigma_1,\ldots,\Sigma_r\in S(\Sigma,\tau)$, $r\geq 1$.
Then
\begin{equation}
\label{eq:intersectionVSigmas}
V_{\Sigma_1}\cap \ldots \cap V_{\Sigma_r} = \big( \bigcap_{j=1}^r \X(\Sigma_j)^H\big) \ojoin \left( \bigcap_{j=1}^r (\OD(\Sigma_j)_{\leq(\tau, \op_{\Sigma_j}(\tau))})^H \right).
\end{equation}
There is at most one opposite of $\tau$ involved this intersection, as the simplices of $\Delta$ appearing therein all lie in $\Sigma_1$.
By Lemma \ref{lm:contractibleIntersectionApartments}, the poset
\[ \bigcap_{j=1}^r \X(\Sigma_j)^H  =  \X\big(\, \bigcap_{j=1}^r (\Sigma_j^H)\, \big)\]
is contractible if it does not contain the opposite of $\tau$, in which case the intersection in Equation \eqref{eq:intersectionVSigmas} is contractible as well.
Thus, we are left with the case that the unique opposite 
$\op_{\Sigma_1}(\tau)$ of $\tau$ in $\Sigma_1$ lies in $\Sigma_j^H$ for all $1\leq j\leq r$.
Here, for all $j$ we have that $(\tau, \op_{\Sigma_1}(\tau)) = (\tau, \op_{\Sigma_j}(\tau))$ is fixed by $H$, and gives a cone point in the right-hand factor of the join in Equation \eqref{eq:intersectionVSigmas}. As a consequence the intersection  $V_{\Sigma_1}\cap \ldots \cap V_{\Sigma_r}$ is contractible.

By Quillen's fiber Theorem \ref{thm:quillen}, we see that $\phi_H:X^H \to Y^H$ is a homotopy equivalence.
By taking $H = 1$ and using the fact that buildings are Cohen-Macaulay, we see that $\OPD(\Delta)^{(i+m+1)} \simeq D_1 \ojoin D_2^{(i)}$ is spherical of dimension $m + i + 1$.
This concludes the proof of item (1).
\end{proof}

From Theorem \ref{thm:OPDandDeltaDelta}, we deduce homotopy formulas for
lower intervals in the poset $\OPD(\Delta)$. For the formulation of our
result we set $\OPD(\Delta)_{\prec (\emptyset, \emptyset)} = 
\OPD(\Delta)$, $\D(\Delta)_{< \emptyset} = \D(\Delta)$, $\PD(\Delta)_{\prec \emptyset} = \PD(\Delta)$, and  $\Conv_\Delta(\emptyset,\emptyset) = \{\emptyset\}$ the empty Levi sphere.

\begin{theorem}
\label{thm:OPDtoPD}
Let $\Delta$ be a spherical building, and let $(\sigma,\sigma')\in \OD(\Delta) \cup \{(\emptyset,\emptyset)\}$.
Let $S := \Conv_{\Delta}(\sigma,\sigma')$.
Then we have a homotopy equivalence
\begin{equation}
\label{eq:OPDwedgePD}
    \OPD(\Delta)_{\prec (\sigma,\sigma')} \simeq \PD(\Delta)_{\prec S} \bigvee_{ T \in {\D(\Delta)_{< S}}} |T| * |T|*\PD(\Lk_{\Delta}(\sigma_T)),
\end{equation}
where $\sigma_T\in T$ is some maximal simplex.
In particular, $\PD(\Delta)$ is spherical of dimension $2\dim\Delta + 1$ and
\begin{equation}
\label{eq:intervalOPD}
\OPD(\Delta)_{\prec (\sigma,\sigma')} \simeq |S| * \OPD(\Lk_{\Delta}(\sigma)).
\end{equation}
\end{theorem}

\begin{proof}
We extend the map $F:\OD(\Delta)\to \D(\Delta)$ given in
\eqref{eq:ForgetfulMap} by the identity map from $\Delta \subseteq \OD(\Delta)$ to $\Delta \subseteq \D(\Delta)$, obtaining a poset map $F:\OPD(\Delta)\to \PD(\Delta)$.
Let $X = \OPD(\Delta)_{\prec (\sigma,\sigma')}$ and $Y = \PD(\Delta)_{\prec S}$.
Now we verify the hypotheses of Theorem \ref{thm:wedgeDecomposition} 
for $X$, $Y$ with $f$ the restriction of $F$ to $X$ with codomain 
$Y$ to deduce the homotopy equivalence in Equation \eqref{eq:OPDwedgePD}.
Indeed, we apply Theorem \ref{thm:wedgeDecomposition} to the opposite
posets, and for that reason we consider upper instead of lower fibers.

We need to distinguish two cases for $y \in Y$.

First, suppose that $y = T\in \D(\Delta)_{\prec S}$.
Then, by projecting onto the first coordinate, we have
\begin{equation}
\label{eq:fiberForgetfulMap}
F^{-1}(Y_{\geq T})  = F^{-1}(\D(\Delta)_{\geq T}\cap \D(\Delta)_{< S}) \cong \{\tau\in T \tq \tau\supsetneq \sigma\} \cong \Lk_{T}(\sigma).
\end{equation}
As $T$ triangulates a sphere, $\Lk_{T}(\sigma)$ is a sphere of dimension $\dim T - \dim \sigma - 1$.
In particular, the inclusion $F^{-1}(Y_{>T}) \hookrightarrow F^{-1}(Y_{\geq T})$ is a null-homotopic map.

Now, let $y = \tau_0 \in \X(\Delta)$ such that $\tau_0 \prec S$.
Then $F^{-1}(Y_{\succeq \tau_0}) = W \cup Z$ where
\begin{align*}
    W & = \{ \tau'\in \X(\Delta) \tq \tau'\supseteq \tau_0, \tau'\prec S\},\\
    Z & = \{ (\tau,\tau')\in \OD(\Delta)_{\prec (\sigma,\sigma')} \tq \tau_0 \prec \Conv_{\Delta}(\tau,\tau') \}.
\end{align*}
In this case $\tau_0$ is the unique minimal element of this preimage, which is therefore a contractible poset.
It follows also that the inclusion $F^{-1}(Y_{\succ \tau_0}) \hookrightarrow F^{-1}(Y_{\succeq \tau_0})$ is null homotopic.

Thus, by Theorems \ref{thm:wedgeDecomposition} and \ref{thm:PDandCB}, we get the wedge decomposition as stated in Equation \eqref{eq:OPDwedgePD}.
By Theorem \ref{thm:OPDandDeltaDelta}, $\PD(\Delta)$ is spherical of dimension $2\dim \Delta + 1$.
Finally, Equation \eqref{eq:intervalOPD} follows from Theorem \ref{thm:PDandCB}(4) and Lemma \ref{lem:lowerIntervalsDecompDelta}.
\end{proof}

\begin{remark}
Let $(\sigma,\sigma')\in \OD(\Delta)$ and $S := \Conv_{\Delta}(\sigma,\sigma')$.
Let $Y := \{\sigma\} \ojoin \OPD(\Lk_{\Delta}(\sigma))$.
It seems that the map $p_{\sigma} : \OPD(\Delta)_{\prec (\sigma,\sigma')} \to Y$ that takes a simplex $\tau\prec (\sigma,\sigma')$ to $\proj_{\sigma}(\tau) \in \{\sigma\} \ojoin \X(\Lk_{\Delta}(\sigma))$, and a pair of opposite simplices $(\tau,\tau')\in \OD(\Delta)_{\prec (\sigma,\sigma')}$ to $(\tau,\proj_{\sigma}(\tau')) \in \OD(\Lk_{\Delta}(\sigma))$ should satisfy the following conditions:
\begin{enumerate}
    \item $p_{\sigma}^{-1}(\sigma) = \X(S)$.
    \item For $\tau\in\X(\Lk_{\Delta}(\sigma))$, $p_{\sigma}^{-1}(Y_{\subseteq \tau})$ is contractible since it is the convex hull of $\tau,\tau'$, where $\tau' = \proj_{\sigma'}(\tau)$.
    \footnote{The equality $p_{\sigma}^{-1}(Y_{\subseteq \tau}) = \Conv_{\Delta}(\tau,\tau')$ was suggested by Bernhard Mühlherr in personal communication.}
    In particular, since $\tau,\tau'$ are not opposite, this is contractible.
    \item For $(\tau,\tau')\in \OD(\Lk_{\Delta}(\sigma))$, $p_{\sigma}^{-1}(Y_{\preceq (\tau,\tau')})$ has a unique maximal element $(\tau,\tau'')$, where $\tau'' = \proj_{\sigma'}(\tau'')$.
    In particular, this is contractible.
\end{enumerate}
However, we were not able to prove items (2) and (3).
A proof of these two items would give an alternative proof that $\OPD(\Delta)_{\prec (\sigma,\sigma')}$ has the homotopy type of $\X(S) \ojoin \OPD(\Lk_{\Delta}(\sigma))$ using Theorem \ref{thm:wedgeDecomposition}.
\end{remark}

The upper-intervals $\PD(\Delta)_{\succ \sigma}$ and $\OPD(\Delta)_{\succ \sigma}$ seem harder to control.
We conjecture that they should behave nicely, in the sense that they are spherical.

\begin{conjecture}
Let $\Delta$ be a spherical building, and let $\sigma \in\Delta$ be any non-empty simplex.
Then the upper intervals $\OPD(\Delta)_{\succ \sigma}$ and $\PD(\Delta)_{\succ \sigma}$ are spherical of dimension $2\dim \Delta -\dim \sigma$.
In particular, $\OPD(\Delta)$ and $\PD(\Delta)$ are Cohen-Macaulay of dimension $2\dim \Delta + 1$.
\end{conjecture}

\section{\texorpdfstring{$\OPD$}{OPD} and \texorpdfstring{$\PD$}{PD} posets for algebraic groups}
\label{sec:algebraicgroups}

In this section, we specialize the construction of decomposition and partial decomposition posets to buildings arising from rational points of arbitrary connected reductive groups.
Our main references are \cite{Borel, BT}. We sometimes 
refrain from quoting results from the preceding sections and
instead present the arguments again in the language of groups. 
We do so in order to facilitate for readers coming from group theory the
identification of the structures. We point to the 
respective results from the preceding section when we
provide these parallel proofs.

Throughout this section, $\GG$ will denote a connected reductive linear group defined over the algebraic closure of some field $k$, and which is split over $k$ (this means that $\GG$ contains a maximal split $k$-torus, see below for definitions).
We write $\GG(k)$ for the set of $k$-points of $\GG$.
Every finite group of Lie type is of the form $\GG(k)$ except for the ``very twisted" groups, namely $^2B_2(q)$, $^2G_2(q)$, and $^2F_4(q)$ (see \cite[1.6.2]{GM}).
The latter can be obtained as fixed-point subgroups of Steinberg endomorphisms.
In this section, however, we will focus on $k$-rational points of algebraic groups.
Analogous results hold as well for pairs $(\GG,F)$ when $\GG$ is defined over the algebraic closure of a finite field and $F$ is a Steinberg endomorphism of $\GG$. The reader can verify that most of our assertions in this section extend to the setting of Steinberg endomorphisms.

For an arbitrary algebraic group $H$, we write $R_u(H)$ for the unipotent radical of $H$.
A Borel subgroup of $\GG$ is a closed, connected, solvable subgroup that is maximal with respect to all these properties.
Borel subgroups always exist, and they are all $\GG$-conjugate.
A parabolic subgroup of $\GG$ is any subgroup of $\GG$ containing a Borel subgroup.
A Levi complement in a parabolic subgroup $P$ is a complement to $R_u(P)$.
That is, a subgroup $L\leq P$ such that $P = R_u(P)\rtimes L$.
By a Levi subgroup of $\GG$ we mean a Levi complement in some parabolic subgroup.
In what follows, when we speak about parabolic or Levi subgroups, we will always mean proper subgroups, unless we specifically say otherwise.
If we want to emphasize that a given Levi subgroup is the Levi complement in a given (parabolic) subgroup, we will explicitly say the word ``complement".
For instance, if $P$ is a parabolic subgroup of $\GG$ and $L$ is a Levi complement in $P$, then we say that $L$ is a Levi subgroup of $Q$, for any $P \subseteq Q \subseteq \GG$.

By a $k$-subgroup of $\GG$ we mean a subgroup $H$ of $\GG$ that is defined over $k$, and $H(k)$ denotes the set of $k$-points of $H$, that is, the subgroup of $H$ consisting of those elements that are defined over $k$.
A maximal split $k$-torus of $\GG$ is a maximal torus $T$ of $\GG$ that is defined over $k$ and is contained in some Borel subgroup of $\GG$ also defined over $k$.
The assumption on $\GG$ being split over $k$ assures the existence of such a $T$.

It is well-known that the $k$-point subgroup $\GG(k)$ has a BN-pair and hence an associated building, which we denote by $\Delta(\GG,k)$.

\begin{theorem}
The group $\GG(k)$ has a BN-pair $(B(k),N(k))$ where $B$ is a $k$-Borel subgroup of $\GG$, and $N = N_{\GG}(T)$ is the normalizer of a maximal split $k$-torus $T$ of $B$.
\end{theorem}

Recall that the simplices of $\Delta(\GG,k)$ are exactly the parabolic subgroups of $\GG$ that are defined over $k$, which we shall call $k$-parabolic subgroups.
The face poset $\X(\Delta(\GG,k))$ is then the poset of (proper) $k$-parabolic subgroups of $\GG$ with order given by reverse inclusion.
A $k$-Levi subgroup is a Levi subgroup of $\GG$ that is defined over $k$ and is the Levi complement in some $k$-parabolic subgroup.
The unipotent radical of a $k$-parabolic subgroup is also defined over $k$.

Recall that any Levi subgroup of $\GG$ is itself a connected reductive subgroup.
Hence, we want to relate Levi subgroups and parabolic subgroups of a fixed Levi subgroup of $\GG$, with those of $\GG$.
The following proposition provides information on this relation.

\begin{proposition}
\label{prop:Levi_and_subLevi}
Let $\GG$ be a connected reductive algebraic group over $k$, and let $P\in \Delta(\GG,k)$.
\begin{enumerate}
    \item Every maximal split $k$-torus in $P$ is contained in a unique Levi complement in $P$.
    Moreover, such Levi complements are defined over $k$ and are regularly permuted by $R_u(P)(k)$.
    \item If $P,Q\in \Delta(\GG,k)$ and $Q\subseteq P$, then for any $k$-Levi complement $M$ in $Q$ there is a unique $k$-Levi complement $L$ in $P$ such that $M\subseteq L$.
    \item If $L$ is a $k$-Levi subgroup of $\GG$ and $M\subseteq L$, then $M$ is a $k$-Levi subgroup of $L$ if and only if $M$ is a $k$-Levi subgroup of $\GG$.
\end{enumerate}
\end{proposition}

\begin{proof}
Item (1) follows from \cite[3.14]{BT}.

Items (2) and (3) follow from Proposition 4.1 of \cite{DM} by noting that the proof of such a proposition extends to arbitrary algebraic closed fields.
The $k$-rational version follows from item (1) as $L$ and $M$ share a maximal split $k$-torus.
\end{proof}

The following result is well-known, and it relates the opposition relation in $\Delta(\GG,k)$ with intrinsic algebraic properties of parabolic subgroups.

\begin{lemma}
\label{lm:oppositionAndParabolics}
Two $k$-parabolic subgroups are opposite in $\Delta(\GG,k)$ (as simplices there) if and only if their intersection is a $k$-Levi complement in both.
Moreover, given $P\in \Delta(\GG,k)$, there is a one-to-one correspondence between $k$-parabolic subgroups opposite to $P$ and $k$-Levi complements in $P$.
The isomorphism is given by $(P,Q)\in \OD(\Delta(\GG,k)) \mapsto P\cap Q$.
In particular, $R_u(P)(k)$ acts regularly on the set of $k$-parabolic subgroups opposite to $P$.
\end{lemma}

\begin{proof}
See \cite[20.5]{Borel} and \cite[4.8]{BT}.
\end{proof}

By the previous lemma, we see that $\OD(\Delta(\GG,k))$ is naturally isomorphic to the poset of pairs $(P,L)$ where $P$ is a $k$-parabolic subgroup and $L$ a $k$-Levi complement of $P$.

On the other hand, as noticed by Serre \cite[3.1.7]{Serre}, $k$-Levi subgroups correspond exactly to Levi spheres of $\Delta(\GG,k)$. 
Below we give a more concrete description of this connection.

\begin{lemma}
[Serre]
\label{lm:LeviSubgroupsAndLeviSpheres}
The map $L \mapsto \Delta(\GG,k)^L$ defines an order-preserving isomorphism from the poset of $k$-Levi subgroups of $\GG$ ordered by inclusion, onto the poset $\D(\Delta(\GG,k))$.
This map is $\GG(k)$-equivariant.
\end{lemma}

\begin{proof}
Note that the maximal simplices of $\Delta(\GG,k)^L$ are exactly the set of $k$-parabolic subgroups of $\GG$ in which $L$ is a Levi complement by Proposition \ref{prop:Levi_and_subLevi}(1). Moreover,
\begin{equation}
\label{eq:fixedPointLevi}
 \Delta(\GG,k)^L = \{ P \in \Delta(\GG,k) \tq P\supseteq L\}
\end{equation}
since parabolic subgroups are self-normalizing.
This shows that $\Delta(\GG,k)^L$ is a subcomplex of $\Delta(\GG,k)$.
In particular, for every maximal simplex $P\in \Delta(\GG,k)^L$ there exists a unique opposite $Q\in \Delta(\GG,k)^L$ such that $P\cap Q = L$ by Lemma \ref{lm:oppositionAndParabolics}.
In fact, as every parabolic in $\Delta(\GG,k)^L$ contains $L$, which contains a maximal split $k$-torus, $\Delta(\GG,k)^L$ is a subcomplex of some apartment.
This means that every simplex of $\Delta(\GG,k)^L$ has exactly one opposite there.
Moreover, $\Delta(\GG,k)^L$ is convex as $L$ must fix every geodesic between two non-opposite points in $\Delta(\GG,k)^L$ (note that here being convex is equivalent to the geometric notion of convexity in terms of geodesics of the metric space $|\Delta(\GG,k)|$ with the CAT(1) metric; cf.  \cite[2.3.1]{Serre}). 
By Lemma \ref{lm:convexIsleviSphere}, $\Delta(\GG,k)^L$ is a Levi sphere.
This proves that $\Delta(\GG,k)^L \in \D(\Delta(\GG,k))$.
Also, the map sending $L$ to $\Delta(\GG,k)^L$ is clearly order-preserving.
By Equation \eqref{eq:fixedPointLevi}, $L$ can be recovered as a Levi complement of a maximal simplex in $\Delta(\GG,k)^L$, showing that this map is a poset embedding.

Finally, for a Levi sphere $S \in \D(\Delta(\GG,k))$, we see that $S = \Delta(\GG,k)^L$ where $L$ is the $k$-Levi complement of a maximal simplex $P\in S$ for which $L$ contains a maximal split $k$-torus giving rise to an apartment containing $S$.
\end{proof}

Hence, we see that $\D(\Delta(\GG,k))$ is isomorphic with the poset of $k$-Levi subgroups ordered by inclusion.

In view of these observations, we can now describe the posets $\PD$, $\D$, $\OPD$ and $\OD$ purely in terms of $k$-subgroups of $\GG$.
Recall that a set of $k$-parabolic subgroups is a subset of simplices of some apartment of $\Delta(\GG,k)$ (i.e., its elements lie in a common apartment), if and only if these parabolic subgroups share a maximal split $k$-torus.

\begin{definition}
\label{def:DecompPosetsAlgebraicGroups}
For $\GG$ a connected reductive group over a field $k$, we write:
\begin{itemize}
    \item $\Delta(\GG,k) =$ the building of $\GG(k)$, whose face poset is the poset of $k$-parabolic subgroups ordered by reverse inclusion.    
    \item $\D(\GG,k) = $ poset of $k$-Levi subgroups, ordered by inclusion.
    \item $\OD(\GG,k) = \{(P,L) \tq P\in \Delta(\GG,k)$ and $L\in \D(\GG,k)$ is a $k$-Levi complement in $P\}$, with ordering given by $(P,L)\leq (P',L')$ if and only if $P\subseteq P'$ and $L\subseteq L'$.
    \item $\PD(\GG,k) = \X(\Delta(\GG,k)) \cup \D(\GG,k)$, with crossed-term ordering given by $P\prec L$ if $P,L$ share a maximal split $k$-torus.
    \item $\OPD(\GG,k) = \X(\Delta(\GG,k)) \cup \OD(\GG,k)$, with crossed-term ordering given by $Q\prec (P,L)$ if $Q,P,L$ share a maximal split $k$-torus.
\end{itemize}
\end{definition}

The following proposition describes the links in the building $\Delta(\GG,k)$.

\begin{proposition}
\label{prop:Levi_and_subParab}
If $(P,L)\in \OD(\GG,k)$, then we have an $L(k)$-equivariant poset isomorphism
\begin{align*}
\X(\Delta(L,k)) & \longrightarrow \X(\Delta(\GG,k))_{\subset P}\\
Q & \longmapsto QR_u(P),
\end{align*}
with inverse given by $Q\mapsto Q\cap L$.
In particular,
\[ \Lk_{\Delta(\GG,k)}(P) \cong  \X(\Delta(\GG,k))_{\subset P} \cong \Delta(L,k).\]
\end{proposition}

\begin{proof}
This follows from Propositions 14.22 and 21.13 of \cite{Borel}.
\end{proof}

Next, we study intervals in the posets $\D$, $\PD$, $\OD$ and $\OPD$ in terms of the algebraic description we have for their elements.

\begin{corollary}
\label{coro:intervalsOD}
Let $(P,L)\in \OD(\GG,k)$.
Then
\begin{enumerate}
\item $\OD(\GG,k)_{>(P,L)} \cong_L (\{ Q\in \Delta(\GG,k) \tq Q\supset P\},\subseteq) \cong_L \X(\Delta(\GG,k))_{\supset P}^{\op}$.
\item $\OD(\GG,k)_{< (P,L)} \cong_{L(k)} \OD(L,k)$, and the isomorphism is given by $(Q,M)\mapsto (Q\cap L,M)$ with inverse $(\widetilde{Q},M)\mapsto (\widetilde{Q}R_u(P),M)$.
\end{enumerate}
\end{corollary}

\begin{proof}
The $L$-isomorphisms in item (1) follow from Proposition \ref{prop:Levi_and_subLevi}(2).

Write $U:=R_u(P)$.
We claim that we have an $L(k)$-isomorphism:
\begin{align*}
    \OD(L,k) & \longrightarrow \OD(\GG,k)_{<(P,L)}\\
    (\widetilde{Q},M) & \longmapsto (\widetilde{Q}U, M).
\end{align*}
Let $(\widetilde{Q},M)\in \OD(L,k)$.
Then $\widetilde{Q}U$ is a $k$-parabolic subgroup properly contained in $P$ by Proposition \ref{prop:Levi_and_subParab}, and $M$ is a $k$-Levi subgroup of $\GG$ by Proposition \ref{prop:Levi_and_subLevi}(3).
Then $M$ is a Levi complement for $\widetilde{Q}U$ since $R_u(\widetilde{Q}U) = R_u(\widetilde{Q})U$.
This shows that $(\widetilde{Q}U,M) \in \OD(\GG,k)_{<(P,L)}$.
Moreover, the inverse of this map is given by $(Q,M) \mapsto (Q\cap L,M)$ by Propositions \ref{prop:Levi_and_subLevi}(3) and \ref{prop:Levi_and_subParab}.
This proves item (2).
\end{proof}

Now let us look at the map $F_{\GG,k}:\OPD(\GG,k)\to \PD(\GG,k)$ that projects a tuple $(P,L)\in \OD(\GG,k)$ to the Levi subgroup $L  \in \D(\GG,k)$, and is the identity on the building.
In particular, if $Q$ is the unique opposite of $P$ containing $L$, then the Levi sphere spanned by $P,Q$ is $\Delta(\GG,k)^L$.
Therefore, our identifications show that $F_{\GG,k}$ is indeed the map $F: \OPD(\Delta(\GG,k)) \to \PD(\Delta(\GG,k))$ from Equation \eqref{eq:ForgetfulMap}.

\begin{proposition}
\label{prop:intervalsDecompositions}
Let $n_{\GG,k}$ denote the dimension of $\Delta(\GG,k)$, and let $L\in \D(\GG,k)$ be a $k$-Levi subgroup.
Then the following hold:
\begin{enumerate}
\item $\D(\GG,k)_{\subset L} = \D(L,k)$ and $\D(\GG,k)_{\supseteq L} \cup \{\GG\}$ is a geometric lattice.
\item We have $\Delta(\GG,k)^L = \{Q\in \Delta(\GG,k) \tq Q\supseteq L\} \cong F_{\GG,k}^{-1}(\D(\GG,k)_{\supseteq L})$.
\item $\Delta(\GG,k)^L$ is a sphere of dimension $n_{\GG,k}-n_{L,k}-1$.
\end{enumerate}
\end{proposition}

\begin{proof}
Item (1) follows from Propositions \ref{prop:Levi_and_subLevi} and \ref{prop:upperIntervalsDecompDelta}.

In item (2), the first equality follows from the fact that parabolic subgroups are self-normalizing.
We show next the isomorphism with $F_{\GG,k}^{-1}(\D(\GG,k)_{\supseteq L})$ in item (2).

By Lemma \ref{lm:LeviSubgroupsAndLeviSpheres}, $T := \Delta(\GG,k)^L$ is the Levi sphere defined by any pair of opposite $k$-parabolic subgroups with Levi complement $L$.
On the other hand, by Equations \eqref{eq:ForgetfulMap} and  \eqref{eq:fiberForgetfulMap} with $\sigma=\emptyset$ there and $T := \Delta(\GG,k)^L$, and the isomorphisms $\OD(\GG,k)\cong \OD(\Delta(\GG,k))$ (Lemma \ref{lm:oppositionAndParabolics}) and $\D(\GG,k)\cong \D(\Delta(\GG,k))$ (Lemma \ref{lm:LeviSubgroupsAndLeviSpheres}), we see that
\[ F_{\GG,k}^{-1}(\D(\GG,k)_{\supseteq L}) \cong_L \Lk_{T}(\sigma) =  \Delta(\GG,k)^L.\]

Item (3) again follows from Lemmas \ref{lm:convexIsleviSphere} and \ref{lm:LeviSubgroupsAndLeviSpheres}.
\end{proof}

\begin{corollary}
The poset $\D(\GG,k)$ is Cohen-Macaulay of dimension $n_{\GG,k}$, and $F_{\GG,k}:\OD(\GG,k)\to \D(\GG,k)$ induces a homotopy equivalence:
\begin{equation}
\label{eq:equivOppD}
\OD(\GG,k) \simeq \D(\GG,k) \vee \bigvee_{L\in \D(\GG,k)} \Delta(\GG,k)^L * \D(L,k).
\end{equation}
Moreover, if $\GG(k)$ is finite, then in homology we get an isomorphism of $\GG(k)$-modules:
\begin{align*}
\redH_{n_{\GG,k}}(\OD(\GG,k),R) & =  \redH_{n_{\GG,k}} (\D(\GG,k),R) \, \oplus\\
&\bigoplus_{\overline{L} \in \D(\GG,k) / \GG(k) } \Ind_{N_{\GG}(L)(k)}^{\GG(k)}\left( \ \redH_{n_{\GG,k}-n_{L,k}-1}(\Delta(\GG,k)^L,R) \otimes \redH_{n_{L,k}}(\D(L,k),R) \ \right),
\end{align*}
where $\overline{L}$ denotes the $\GG(k)$-orbit of $L\in \D(\GG,k)$.
Here, $R$ is a ring such that $R\GG(k)$ is semisimple.
\end{corollary}

\begin{proof}
The first part follows from Theorem \ref{thm:DisCM}.
The homology isomorphism is an application of Theorem \ref{thm:homologyWedge} and the description of the fibers and intervals from Proposition \ref{prop:intervalsDecompositions}.
For a $k$-Levi subgroup $L\in \D(\GG,k)$, note that $\Stab_{\GG(k)}(L) = N_{\GG}(L)(k)$.
\end{proof}

Now, we provide a group-theoretic description of the lower intervals in the partial decomposition posets $\PD$ and $\OPD$, and the connecting map between them.

\begin{proposition}
\label{prop:lowLeviIntervalOPD}
Let $(P_L,L)\in \OPD(\GG,k)$.
Then $\OPD(\GG,k)_{\prec (P_L,L)} \simeq \Delta(\GG,k)^L * \OPD(L,k)$. 
Moreover, in homology, we have an isomorphism of $L(k)$-modules
\begin{equation}
\label{eq:isoHomologyIntervalsOPDLevi}
\widetilde{H}_*\big(\OPD(\GG,k)_{\prec (P_L,L)},R \big) \cong \widetilde{H}_*\big(\, \Delta(\GG,k)^L *\OPD(L,k), R \, \big) \cong \widetilde{H}_*\big(\, \OPD(L,k), R \, \big),
\end{equation}
provided that $RL(k)$ is semisimple.
\end{proposition}

\begin{proof}
Let $X_L := \OPD(\GG,k)_{\prec (P_L,L)}$, and $Y_L := \{L\} \ojoin \OPD(L,k)$, where $L$ is a unique minimal element of $Y_L$.
By Corollary \ref{coro:intervalsOD}, we have an $L(k)$-isomorphism of posets $\OD(\GG,k)_{<(P_L,L)} \cong_{L(k)} \OD(L,k)$ given by intersecting,
\[(Q,M)\in \OD(\GG,k)_{<(P_L,L)}\mapsto (Q\cap L, M)\in \OD(L,k).\]
Consider then the ``extended intersection map":
\begin{align*}
I_L : X_L & \longrightarrow Y_L\\
      (Q,M) \in \OD(\GG,k)_{<(P_L,L)} & \mapsto (Q\cap L,M)\in \OD(L,k)\\
      Q \in \X(\Delta(\GG,k))_{\prec (P_L,L)} & \mapsto Q\cap L \in \{L\} \cup \X(\Delta(L,k)).
\end{align*}
Now $Q\cap L\in \{L\}\cup \X(\Delta(L,k))$ by \cite[21.13(i)]{Borel}.
Note that the image of $I_L$ contains the unique minimal element $L$ of the poset $Y_L$.

We study the lower fibers of $I_L$.
If $P\in \Delta(L,k)$, and $Q\in \Delta(\GG,k)$ contains $P$, then $Q$ and $L$ share a maximal split $k$-torus as $P$ contains a maximal split $k$-torus and $P\subset L$.
Therefore,
\[ I_L^{-1}( (Y_L)_{\leq P} ) = \{ Q \in \Delta(\GG,k) \tq Q\cap L \supseteq P\} = \{ Q\in \Delta(\GG,k) \tq Q\supseteq P\} = \Delta(\GG,k)^P.\]
Since $P\neq L$, $R_u(P)\neq 1$, we see that $\Delta(\GG,k)^P$ is contractible by Proposition 4.1 of \cite{Serre}.
As $P$ acts trivially on this complex, the contraction is clearly $P$-equivariant (note that $N_L(P) = P$ as $P$ is a parabolic subgroup of $L$).

On the other hand,
\[ I_L^{-1}( (Y_L)_{\leq L} ) = I_L^{-1}(L) = \{ Q \in \Delta(\GG,k) \tq Q\supseteq L\} = \Delta(\GG,k)^L,\]
which is a sphere of dimension $n_{\GG,k}-n_{L,k}-1$ by Proposition \ref{prop:intervalsDecompositions}.  Moreover, $L$ acts trivially on $\Delta(G,k)^L$.

Finally, for $(Q,M)\in \OD(L,k)$, $I_L^{-1}( (Y_L)_{\leq (Q,M)} )$ has a unique maximal element $(Q_L,M)$, where $Q_L = QR_u(P_L)$ is the unique parabolic subgroup contained in $P_L$ and such that $Q_L\cap L = Q$ (see Corollary \ref{coro:intervalsOD}).
Hence, this fiber is $M$-contractible, where $M = N_L(M) \cap N_L(Q)$.

By Theorems \ref{thm:wedgeDecomposition} and \ref{thm:homologyWedge}, we conclude that $I_L$ induces a homotopy equivalence between $X_L$ and $\Delta(\GG,k)^L * \OPD(L)$, which gives the first $L(k)$-isomorphism in homology as stated in Equation \eqref{eq:isoHomologyIntervalsOPDLevi}.
The second isomorphism follows since $L(k)$ acts trivially on $\Delta(G,k)^L$.
\end{proof}

We obtain an analogue result for $\PD$:

\begin{proposition}
\label{prop:lowLeviIntervalPD}
Let $L\in \D(\GG,k)$.
Then $\PD(\GG,k)_{\prec L} \simeq \Delta(\GG,k)^L * \PD(L,k)$. 
Moreover, in homology, we have an isomorphism of $L(k)$-modules
\begin{equation}
\label{eq:isoHomologyIntervalsPDLevi}
\widetilde{H}_*\big(\PD(\GG,k)_{\prec L},R \big) \cong \widetilde{H}_*\big(\, \Delta(\GG,k)^L *\PD(L,k), R \, \big) \cong \widetilde{H}_*\big(\,\PD(L,k), R \, \big),
\end{equation}
provided that $RL(k)$ is semisimple.
\end{proposition}

\begin{proof}
We proceed as we did the proof of Proposition \ref{prop:lowLeviIntervalOPD}, and consider the intersection map on the $\Delta(\GG,k)$-part.

Let $W_L = \PD(\GG,k)_{<L}$, and $Z_L = \{L\}\ojoin \PD(L,k)$.
Since $\D(\GG,k)_{\prec L} = \D(L,k)$ by Proposition \ref{prop:intervalsDecompositions}, the intersection map $J_L:W_L\to Z_L$ is the identity on the elements of $\D(\GG,k)_{<L}$, and intersects with $L$ those parabolics $Q\in \Delta(\GG,k)$ which share a maximal split $k$-torus with $L$.
The same argument as in the previous proof shows that for $P\in \Delta(L,k)$, one has that
\[ J_L^{-1}( (Z_L)_{\leq P} ) = I_L^{-1}( (Y_L)_{\leq P} ) = \Delta(\GG,k)^P \simeq_{P} * .\]

On the other hand, the preimage $J_L^{-1}(L)$ is again the sphere $\Delta(\GG,k)^L$ of dimension $n_{\GG,k}-n_{L,k}-1$, with trivial action of $L$.

Finally, the preimage $J_L^{-1}( (Z_L)_{\preceq M} )$, where $M\in \D(L,k)$, is $N_{L}(M)$-contractible since it has a unique maximal element $M$.

As before, invoking Theorems \ref{thm:wedgeDecomposition} and \ref{thm:homologyWedge} we conclude the proof.
As in the proof of Proposition \ref{prop:lowLeviIntervalOPD}, the second isomorphism in Equation \eqref{eq:isoHomologyIntervalsPDLevi} follows since $L(k)$ acts trivially on $\Delta(G,k)^L$.
\end{proof}

To finish this section, we provide a long exact sequence relating the squares of the Steinberg modules of the $k$-Levi subgroups.
In what follows, $\St(\GG(k))$ denotes the homology of $\Delta(\GG,k)$.

\begin{proposition}
Let $\GG$ be a connected reductive group defined over $k$.
Then we have an isomorphism $\widetilde{H}_*(\OPD(\GG,k)) \cong_{\GG(k)} \St(\GG(k))^{\otimes 2}$.
Moreover, we have a long exact sequence of $\GG(k)$-modules

\begin{equation}\label{eq:LESSteinbergSquare}~
\begin{tikzcd}[row sep=3ex, column sep=3ex]
0 \arrow{d}{} & & & \\
 \St(\GG(k))^{\otimes 2} \arrow{d}{} &    & &  \\
\displaystyle{\bigoplus_{\overline{(P,L)} \in O^{(m)}}} \Ind_{L(k)}^{\GG(k)}\big( \St(L(k))^{\otimes 2} \big) \arrow{r}{} & \cdots \arrow{r}{} & \displaystyle{\bigoplus_{\overline{(P,L)} \in O^{(0)}}} \Ind_{L(k)}^{\GG(k)}\big( \St(L(k))^{\otimes 2} \big) \ar[r] & \St(\GG(k)) \arrow[d]{} \\[-2ex]
 & & & 0
\end{tikzcd} \end{equation}\\
where $m = \dim \Delta(\GG,k)$, $O^{(j)}$ is the set of orbits by the $\GG(k)$-action on the set of pairs $(P,L)\in \OD(\GG,k)$
with $\dim P = m-j$ (regarding $P$ as a simplex), for $0\leq j\leq m$, and $\overline{(P,L)}$ denotes the $\GG(k)$-orbit of $(P,L)$.

In particular, if $G(k)$ is finite, $\St(G(k))$ is the alternating sum
\[(-1)^{m}\St(G(k))^{\otimes 2} - \sum_{k=0}^{m}(-1)^k \sum_{\overline{(P,L)} \in O^{(k)}}  \Ind_{L(k)}^{\GG(k)}\big( \St(L(k))^{\otimes 2} \big)\]
in the Grothendieck ring over $\ZZ$.
\end{proposition}

\begin{proof}
Let $X := \OPD(\GG,k)$.
Then $\widetilde{H}_*(X)$ carries the Steinberg square representation by Theorem \ref{thm:OPDandDeltaDelta}(2).

Also, by Theorem \ref{thm:OPDandDeltaDelta}, $X^{(i)} = \OPD(\GG,k)^{(i)}$ is spherical of dimension $i$ for all $0\leq i \leq 2\dim X + 1$.
In particular, by Quillen's fiber theorem \ref{thm:quillen}, the inclusion $X^{(i)} \hookrightarrow X^{(i+1)}$ is an $i$-equivalence.
Now suppose that $m = \dim \Delta(\GG,k)$.
Then $X^{(m)} = \Delta(\GG,k)$, and $\widetilde{H}_*(X) = \St(\GG(k))$.

On the other hand, for any $i\geq 0$, we have a short exact sequence in (integral) homology:
\begin{equation}
\label{eq:SECskeletons}
0\to \widetilde{H}_{i+1}(X^{(i+1)}) \to \widetilde{H}_{i+1}(X^{(i+1)}, X^{(i)}) \to \widetilde{H}_i(X^{(i)}) \to 0.
\end{equation}
Since $X^{(i+1)}$ is obtained from $X^{(i)}$ by adding elements $x$ of height $i$ through their spherical lower-links, namely $X_{<x}$, we see that
\begin{equation}
\label{eq:relativeHomologySkeletons}
\widetilde{H}_{i+1}(X^{(i+1)}, X^{(i)}) \cong_{\GG(k)} \bigoplus_{x \in X^{(i+1)} \setminus X^{(i)}} \widetilde{H}_i(X_{<x}).
\end{equation}
Hence, Equations \eqref{eq:SECskeletons} and \eqref{eq:relativeHomologySkeletons} give rise to a short exact sequence of $\GG(k)$-modules
\begin{equation*}
\label{eq:SECskeletons2}
0\to \widetilde{H}_{i+1}(X^{(i+1)}) \overset{\partial_i}{\longrightarrow}  \bigoplus_{x \in X^{(i+1)} \setminus X^{(i)}} \widetilde{H}_i(X_{<x}) \overset{\iota_i}{\longrightarrow} \widetilde{H}_i(X^{(i)}) \to 0,
\end{equation*}
where $j_i$ is the map induced by the inclusions $X_{<x}\hookrightarrow X^{(i)}$ for $x\in X$ of height $i+1$ (see Theorems 2.4 and 2.5 of \cite{SegevWebb} for more details on the connecting homomorphism $\partial_i$).

Now, suppose that $x\in X$ is an element of the form $x=(P,L)\in \OD(\GG,k)$.
Then the height of $x$ is $2m-\dim P+1$, and $\dim X_{<x} = 2m - \dim P$, where $\dim P$ denotes the dimension of $P$ as a simplex of the building.
Also,
\begin{equation*}
\label{eq:homologyInterval}
\widetilde{H}_*(X_{<x}) \cong_{L(k)} \widetilde{H}_*\left(\Delta(\GG,k)^L * \OPD(L,k) \right) \cong_{L(k)} \St(L(k))^{\otimes 2}.
\end{equation*}
Here we have used that $L$ acts trivially on the sphere $\Delta(\GG,k)^L$.

Note that $\Stab_{\GG(k)}(x) = P\cap N_{\GG}(L)(k) = L(k)$.
Hence, for $m + 1\leq i\leq 2m+1$, we see that
\begin{equation}
\label{eq:relativeHomologySkeletons2}
\bigoplus_{x \in X^{(i+1)} \setminus X^{(i)}} \widetilde{H}_i(X_{<x}) \cong_{\GG(k)} \bigoplus_{\overline{(P,L)}} \Ind_{L(k)}^{\GG(k)}\big( \St(L(k))^{\otimes 2} \big),
\end{equation}
where the second sum runs over the orbits of pairs $(P,L)\in \OD(\GG,k)$ such that $P$, as a simplex, has dimension $2m-i$.
For $0\leq j \leq m$, let $O^{(j)}:= (X^{(j+m+1)} \setminus X^{(j+m)}) / \GG(k)$.
Then, putting all together, and using that $X^{(2m+1)} = X$ with homology given by $\St(\GG(k))^2$,
we get the long exact sequence of $\GG(k)$-modules as in Equation \eqref{eq:LESSteinbergSquare}.
The connecting maps are, from left to right, $\alpha_j = \partial_{j+m}\circ \widetilde{\iota}_{j+m+1}$, $j=m,\ldots,-1$.
Here $\widetilde{\iota}_i$ is the map induced by the inclusion $\iota_i$ under the identification from Equation \eqref{eq:relativeHomologySkeletons2}.

The final claim of the theorem follows from taking alternating sums in Equation \eqref{eq:LESSteinbergSquare} when $G(k)$ is finite (so the modules are finitely generated).
\end{proof}

\bibliographystyle{plainurl} 
\bibliography{refs.bib}

\end{document}